\documentclass[12pt]{amsart}

\usepackage{graphicx,amssymb,latexsym,amsfonts,txfonts,amsmath,amsthm}
\usepackage{pdfsync,color,tabularx,rotating}
\usepackage[all,cmtip]{xy}

\usepackage{tikz}

\advance\hoffset by -0.9 truecm

\newtheorem{theo}{Theorem}
\newtheorem{prop}[theo]{Proposition}
\newtheorem{coro}[theo]{Corollary}
\newtheorem{lemm}[theo]{Lemma}
\newtheorem{conj}[theo]{Conjecture}
\newtheorem{prob}[theo]{Problem}

\theoremstyle{remark}
\newtheorem{rema}[theo]{\bf Remark}

\theoremstyle{remark}

\theoremstyle{remark}

\begin{document}

\title{Realisation of groups as automorphism groups in categories}

\author{Gareth A. Jones}

\address{School of Mathematical Sciences, University of Southampton, Southampton SO17 1BJ, UK}
\email{G.A.Jones@maths.soton.ac.uk}

\iffalse
%\thanks{Partially supported by Project MTM2016-79497-P, Project 
%SEV-2015-0554,  Project Fondecyt 1150003 and Project Anillo ACT1415 PIA-CONICYT}
%05C10 Planar gphs, top aspects of gph thy
%14H57 Dessins d'enfants thy
%20B05 Perm gps, gen thy for finite gps
%20B07 Perm gps, gen thy for infinite gps
%20B25 Fin auto gps of alg, geom or comb structures
%20B27 Infinite automorphism groups
%51M20 Polyhedra & polytopes, ...
%52B11 n-dim polytopes
%52B15 Sym properties of polytopes
%57M10 Covering spaces
\fi
\maketitle

%%%%%%%%%%%%%%%%%
%%%%%%%%%%%%%%%%%

\begin{abstract}
It is shown that in various  categories, including many consisting of maps or hypermaps, oriented or unoriented, of a given hyperbolic type, every countable group $A$ is isomorphic to the automorphism group of uncountably many non-isomorphic objects, infinitely many of them finite if $A$ is finite. In particular, the latter applies to dessins d'enfants, regarded as finite oriented hypermaps. The proof, involving maximal subgroups of various triangle groups, yields a simple construction of a regular map whose automorphism group contains an isomorphic copy of every finite group.
\end{abstract}

%%%%%%%%%%%%%%%%%%%%%%%%%%%%
%%%%%%%%%%%%%%%%%%%%%%%%%%%%

\section{Introduction}\label{intro}

In 1939 Frucht published his celebrated theorem~\cite{Fru} that every finite group is isomorphic to the automorphism group of a finite graph; in 1960, by allowing infinite graphs, Sabidussi~\cite{Sab} extended this result to all groups. Similar results have been obtained, realising all finite groups (or in some cases all groups) as automorphism groups of various other mathematical structures. Examples include the following, in chronological order: distributive lattices, by Birkhoff~\cite{Bir46} in 1946;  regular graphs of a given degree, by Sabidussi~\cite{Sab57} in 1957; Riemann surfaces, by Greenberg~\cite{Gre} in 1960; projective planes, by Mendelsohn~\cite{Men72} in 1972; Steiner systems, by Mendelsohn~\cite{Men78} in 1978; fields, by Fried and Koll\'ar~\cite{FK} in 1978; matroids of rank 3, by Babai~\cite{Bab81} in 1981; oriented maps and hypermaps, by Cori and Mach\`\i~\cite{CM} in 1982; finite volume hyperbolic manifolds of a given dimension, by Belolipetsky and Lubotzky~\cite{BeLu} in 2005; abstract polytopes, by Schulte and Williams~\cite{SW} in 2015 and by Doignon~\cite{Doi} in 2016. Babai has given comprehensive surveys of this topic in~\cite{Bab81, Bab95}.

In many of these cases, each group is represented as the automorphism group of not just one, but infinitely many non-isomorphic objects. The aim of this paper is to obtain results of this nature for certain `permutational categories', introduced and discussed in~\cite{Jon16}. These are
categories $\mathfrak C$ with a `parent group' $\Gamma=\Gamma_{\mathfrak C}$ such that each object $\mathcal O$ in  $\mathfrak C$ can be identified with a permutation representation $\theta:\Gamma\to S:={\rm Sym}(\Omega)$ of $\Gamma$ on some set $\Omega$, and the morphisms ${\mathcal O}_1\to{\mathcal O_2}$ can be identified with the functions $\Omega_1\to\Omega_2$ which commute with the actions of $\Gamma$ on the corresponding sets $\Omega_i$. They include the categories of maps or hypermaps on surfaces, oriented or unoriented, and possibly of a given type; other examples include the category of coverings of a `suitably nice' topological space; this includes the category of dessins d'enfants, regarded as finite coverings of the thrice-punctured sphere, or equivalently as finite oriented hypermaps. 

The automorphism group ${\rm Aut}_{\mathfrak C}(\mathcal O)$ of an object $\mathcal O$ in a permutational category $\mathfrak C$ is identified with the centraliser $C:=C_S(G)$ in $S$ of the monodromy group $G:=\theta(\Gamma)$ of $\mathcal O$. Now $\mathcal O$ is connected if and only if $G$ is transitive on $\Omega$, as we will assume throughout this paper. Such objects correspond to conjugacy classes of subgroups of $\Gamma$, the point-stabilisers. An important result is the following:

\begin{theo}\label{isothm}
${\rm Aut}_{\mathfrak C}({\mathcal O})\cong N_G(H)/H\cong N_{\Gamma}(M)/M$,
where $H$ and $M$ are the stabilisers in $G$ and $\Gamma$ of some $\alpha\in\Omega$, and $N_G(H)$ and $N_{\Gamma}(M)$ are their normalisers. 
\end{theo}

There are analogous results in various contexts, ranging from abstract polytopes to covering spaces, which can be regarded as special cases of Theorem~\ref{isothm}. Proofs of this result for particular categories can be found in the literature: for instance, in~\cite{JS} it is deduced for oriented maps from a more general result about morphisms in that category; in~\cite[Theorem~2.2 and Corollary~2.1]{JW} a proof for dessins is briefly outlined; similar results for covering spaces are proved in~\cite[Appendix]{Mas} and~\cite[Theorem~81.2]{Mun}, and for abstract polytopes in~\cite[Propositions~2D8 and 2E23(a)]{MS}. Theorem~\ref{isothm} follows immediately from the following `folklore' theorem, proved in~\cite{Jon18a}:

\begin{theo}\label{autothm}
Let $G$ be a transitive permutation group on a set $\Omega$, with $H$ the stabiliser of some $\alpha\in\Omega$, and let $C:=C_S(G)$ be the centraliser of $G$ in the symmetric group $S:={\rm Sym}(\Omega)$. Then $C \cong N_G(H)/H$.
\end{theo}

\iffalse
It is natural to ask whether there are any restrictions on the automorphism groups of objects in a given category, apart from the obvious ones concerning their cardinality.
\fi

Of course, finite objects, in any category, have finite automorphism groups. In most of the permutational categories we are interested in, the parent groups are finitely generated, so the automorphism groups of connected objects are all countable. Let us say that a category $\mathfrak C$ has the {\em countable} (resp.~{\em finite}) {\em automorphism realisation property\/} if every countable (resp.~finite) group $A$ is isomorphic to ${\rm Aut}_{\mathfrak C}({\mathcal O})$ for some connected object (resp.~finite connected object) $\mathcal O$ in $\mathfrak C$. We will abbreviate these properties to the CARP and FARP, and will say that $\mathfrak C$ has the {\em strong automorphism realisation property} (SARP) if it has both of them. In the case of a permutational category $\mathfrak C$, these properties follow immediately from Theorem~\ref{isothm} if the associated parent group $\Gamma$ has the corresponding properties, namely that every countable group $A$ is isomorphic to $N_{\Gamma}(M)/M$ for some subgroup $M$ of $\Gamma$, and this can be chosen to have finite index in $\Gamma$ if $A$ is finite.

We will be mainly concerned with permutational categories consisting of maps and hypermaps of various types $(p,q,r)$, where $p, q, r\in{\mathbb N}\cup\{\infty\}$. For these, the parent groups are either extended triangle groups $\Delta[p,q,r]$, generated by reflections in the sides of a triangle with internal angles $\pi/p$, $\pi/q$ and $\pi/r$, or (for subcategories of oriented objects) their orientation-preserving subgroups, the triangle groups $\Delta(p,q,r)$. We say that a triple $(p,q,r)$ is {\sl spherical, euclidean\/} or {\sl hyperbolic\/} as $p^{-1}+q^{-1}+r^{-1}>1, =1$ or $<1$ respectively (where by convention we take $\infty^{-1}=0$), so that these groups act on the sphere, euclidean plane, or hyperbolic plane. We will use Theorem~\ref{isothm} to prove:

\begin{theo}\label{realisation}
{\rm(a)} If $(p,q,r)$ is a hyperbolic triple, where $p, q, r\in{\mathbb N}\cup\{\infty\}$, then the groups $\Delta(p,q,r)$ and $\Delta[p,q,r]$, together with their associated categories $\mathfrak C$ of oriented hypermaps and of all hypermaps of type $(p,q,r)$, have the finite automorphism realisation property.
\vskip2pt
{\rm(b)} If in addition at least one of $p, q$ and $r$ is $\infty$, then the above groups and categories all have the countable and hence the strong automorphism realisation property.
\vskip2pt
{\rm(c)} In (a), each finite group is realised as ${\rm Aut}_{\mathfrak C}(\mathcal O)$ by $\aleph_0$ non-isomorphic finite connected objects $\mathcal O\in\mathfrak C$, while in (b), each countable group is realised by $2^{\aleph_0}$ non-isomorphic connected objects.
\end{theo}

The cardinalities in Theorem~\ref{realisation}(c) are the best possible, since they are the numbers of finite and of all connected objects, up to isomorphism, in the relevant categories. By contrast with Theorem~\ref{realisation}, if we take $\Gamma$ to be a Tarski monster~\cite{Ols80} then $N_{\Gamma}(M)=M$ for every subgroup $M\ne 1$, so the only groups realised as automorphism groups in the corresponding category are $1$ and $\Gamma$.

The spherical and euclidean triples must be excluded since the corresponding triangle groups are either finite or solvable, so the same restriction applies to the automorphism groups of objects in the associated categories. By taking $p=r=\infty$ and $q=2$ or $\infty$ we see that the categories $\mathfrak M$ and $\mathfrak H$ of all maps and hypermaps, together with their subcategories $\mathfrak M^+$ and $\mathfrak H^+$ of oriented maps and hypermaps, satisfy Theorem~\ref{realisation}. In the case of $\mathfrak M^+$ and $\mathfrak H^+$, Cori and Mach\`\i\/~\cite{CoMa} showed in 1982 that every finite group arises as an automorphism group; they considered only finite groups, but their proof extends to countable groups. In fact, by Theorem~\ref{realisation}(a) the category of Grothendieck's dessins d'enfants~\cite{Gro} of any given hyperbolic type has the FARP. Of course dessins, regarded as finite oriented hypermaps, have finite automorphism groups, so they do not satisfy the CARP. In \S\ref{cocompact} we will prove a result, based on early work of Conder~\cite{Con80, Con81}, to support the following conjecture:

\begin{conj}\label{SARPconj}
The condition that $p, q$ or $r=\infty$ can be omitted from Theorem~\ref{realisation}(b), so that the triangle groups $\Delta(p,q,r)$ and $\Delta[p,q,r]$ of {\sl any\/} hyperbolic type, and their associated categories, satisfy the  strong automorphism realisation property.
\end{conj}

\iffalse
Theorem~\ref{realisation} is an analogue of Frucht's theorem for finite graphs and groups~\cite{Fru} and its extension to the infinite case by Sabidussi~\cite{Sab}, and of a theorem of Greenberg for Riemann surfaces~\cite{Gre}, where the proofs are completely different.
\fi

In fact Theorem~\ref{realisation}(a) can be deduced from a more general and very deep result of Belolipetsky and Lubotzky~\cite[Theorem~2.1]{BeLu}, which implies the FARP for every finitely generated group which has a subgroup of finite index with an epimorphism onto a non-abelian free group. This applies to every non-elementary finitely generated Fuchsian group, and in particular to every hyperbolic triangle group. However, the proof of~\cite[Theorem~2.1]{BeLu} is long, delicate and non-constructive, so here we offer a shorter, more direct argument, specific to the context of this paper.

The proof of Theorem~\ref{realisation} is divided into several cases, depending on the particular triangle group $\Gamma=\Gamma_{\mathfrak C}$ involved. In each case we construct a primitive permutation representation of $\Gamma$, of infinite or unbounded finite degree, such that a point stabiliser $N$ has an epimorphism onto a free group of countably infinite or unbounded finite rank, and hence onto an arbitrary countable or finite group $A$. By arranging that the kernel $M$ is not normal in $\Gamma$ we see from the maximality of $N$ in $\Gamma$ that $N_{\Gamma}(M)/M=N/M\cong A$, so that Theorem~\ref{isothm} gives ${\rm Aut}_{\mathfrak C}({\mathcal O})\cong A$ for the object $\mathcal O\in\mathfrak C$ corresponding to $M$. For part (c) we show that there are $\aleph_0$ or $2^{\aleph_0}$ mutually non-conjugate such subgroups $M\le\Gamma$.

\iffalse
For each group $A$ we can find infinitely many conjugacy classes in $\Gamma$ of such subgroups $M$, giving infinitely many non-isomorphic objects $\mathcal O$ realising $A$. They are all regular covers, with covering group $A$, of the object $\mathcal N\in\mathfrak C$ corresponding to $N$.
\fi

As a by-product of the proof of Theorem~\ref{realisation}(b), we will construct in \S\ref{embedding} a simple example of a regular oriented map such that every finite group is embedded in its automorphism group.

Although there is a very wide variety of finite groups, there are also many phenomena which are displayed by countably infinite groups, but not by finite groups, and which can therefore, by Theorem~\ref{realisation}(b), arise for the automorphism groups of infinite maps and hypermaps, but not of their finite analogues. Examples (all finitely generated and hence countable) include:
\begin{itemize}
\item finitely presented groups with unsolvable word and conjugacy problems (see~\cite{Nov55});
\item Hopfian groups, those  isomorphic to proper quotients of themselves (see~\cite{BS});
\item pairs of non-isomorphic groups, each an epimorphic image of the other (see~\cite{BS});
\item co-Hopfian groups, isomorphic to proper subgroups of themselves ($C_{\infty}$, for example);
\item groups $G\cong G^n\not\cong G^r$ for all $r=2,\ldots, n-1$, where $n>1$ (see~\cite{Tjon80}).
%\item Tarski monsters, infinite groups in which every non-identity proper subgroup has order $p$ for some fixed prime $p$ (see~\cite{Ols80}).
\end{itemize}
These observations generalise some of those made in~\cite{Jon08} for automorphism groups of regular oriented hypermaps.  (A connected object ${\mathcal O}\in{\mathfrak C}$ is called {\sl regular\/} if ${\rm Aut}_{\mathfrak C}({\mathcal O})$ acts transitively on $\Omega$, or equivalently $M$ is a normal subgroup of $\Gamma$, in which case $C\cong G\cong \Gamma/M$.)

One should not confuse the CARP with the SQ-universality of a group $\Gamma$, a concept of largeness introduced by P.~M.~Neumann in~\cite{Neu}: this requires that every countable group is isomorphic to a subgroup of a quotient of $\Gamma$, that is, to $N/M$ where $M\le N\le \Gamma$ and $M$ is normal in $\Gamma$, so that $N_{\Gamma}(M)=\Gamma$, while the CARP differs in requiring that $N_{\Gamma}(M)=N$. In terms of permutational categories, the SQ-universality of the parent group $\Gamma$ means that every countable group is embedded in the automorphism group of some regular object, whereas the CARP means that it is isomorphic to the automorphism group of some object, not necessarily regular.

\medskip

\noindent{\bf Acknowledgments.} The author is grateful to Ernesto Girondo, Gabino Gonz\'alez-Diez and Rub\'en Hidalgo for discussions about dessins d'enfants which motivated this work, and also to Ashot Minasyan, Egon Schulte, David Singerman and Alexander Zvonkin for some very helpful comments.

%%%%%%%%%%%%%%%%%%%%%%%%

\section{Permutational categories}\label{permcats}

Following~\cite{Jon16}, let us define a {\em permutational category\/} $\mathfrak C$ to be a category which is naturally isomorphic to the category of permutation representations $\theta:\Gamma\to S:={\rm Sym}(\Omega)$ of a {\em parent group\/} $\Gamma=\Gamma_{\mathfrak C}$. We then define the {\em automorphism group\/} ${\rm Aut}({\mathcal O})={\rm Aut}_{\mathfrak C}({\mathcal O})$ of an object $\mathcal O$ in $\mathfrak C$ to be the group of all permutations of $\Omega$ commuting with the action of $\Gamma$ on $\Omega$; thus it is  the centraliser $C_S(G)$ of the {\em monodromy group\/} $G=\theta(\Gamma)$ of $\mathcal O$ in the symmetric group $S$. In this paper we will restrict our attention to the {\em connected\/} objects $\mathcal O$ in $\mathfrak C$, those corresponding to transitive representations of $\Gamma$. We will pay particular attention to those categories for which the parent group $\Gamma$ is either an extended triangle group
\[\Delta[p,q,r]=\langle R_0, R_1, R_2\mid R_i^2=(R_1R_2)^p=(R_2R_0)^q=(R_0R_1)^r=1\rangle,\]
or its orientation-preserving subgroup of index $2$, the triangle group
\[\Delta(p,q,r)=\langle X, Y, Z\mid X^p=Y^q=Z^r=XYZ=1\rangle,\]
where $X=R_1R_2$, $Y=R_2R_0$ and $Z=R_0R_1$. Here $p, q, r\in{\mathbb N}\cup\{\infty\}$, and we ignore any relations of the form $W^{\infty}=1$. We will now give some important examples of such categories; for more details, see~\cite{Jon16}. In what follows, $C_n$ denotes a cyclic group of order $n$, $F_n$ denotes a free group of rank $n$, $V_4$ denotes a Klein four-group $C_2\times C_2$ and $*$ denotes a free product.

\smallskip

\noindent{\bf 1.} The category $\mathfrak M$ of all maps on surfaces (possibly non-orientable or with boundary) has parent group
\[\Gamma=\Gamma_{\mathfrak M}=\Delta[\infty, 2, \infty]\cong V_4*C_2.\]
This group acts on the set $\Omega$ of incident vertex-edge-face flags of a map (equivalently, the faces of its barycentric subdivision), with each generator $R_i\;(i=0, 1, 2)$ changing the $i$-dimensional component of each flag (whenever possible) while preserving the other two.

\smallskip

\noindent{\bf 2.} The subcategory $\mathfrak M^+$ of $\mathfrak M$ consists of the oriented maps, those in which the underlying surface is oriented and without boundary. This category has parent group
\[\Gamma=\Gamma_{\mathfrak M^+}=\Delta(\infty, 2, \infty)\cong C_{\infty}*C_2,\]
the orientation-preserving subgroup of index $2$ in $\Delta[\infty,2,\infty]$. This group acts on the directed edges of an oriented map: $X$ uses the local orientation to rotate them about their target vertices, and the involution $Y$ reverses their direction, so that $Z$ rotates them around incident faces. Here, and in the preceding example, $\Delta(p,2,r)$ and $\Delta[p,2,r]$ are the parent groups for the subcategories of maps of type $\{r,p\}$ in the notation of~\cite{CM}, meaning that the valencies of all vertices and faces divide $p$ and $r$ respectively, so that $X^p=Z^r=1$. (By convention, all positive integers divide $\infty$.)

\smallskip

\noindent{\bf 3.} Hypermaps are natural generalisation of maps, without the restriction that each edge is incident with at most two vertices and faces which implies that $Y^2=1$. There are several ways of defining or representing hypermaps. The most convenient way is via the Walsh bipartite map~\cite{Wal}, where the black and white vertices correspond to the hypervertices and hyperedges of the hypermap, the edges correspond to incidences between them, and the faces correspond to its hyperfaces. The category $\mathfrak H$ of all hypermaps (possibly unoriented and with boundary) has parent group
\[\Gamma=\Gamma_{\mathfrak H}=\Delta[\infty, \infty, \infty]\cong C_2*C_2*C_2.\]
This group acts on the incident edge-face pairs of the bipartite map, with $R_0$ and $R_1$ preserving the face and the incident white and black vertex respectively, while $R_2$ preserves the edge. As in the case of maps, $\Delta[p,q,r]$ is the parent group for the subcategory of hypermaps of type $(p,q,r)$.

\smallskip

\noindent{\bf 4.}  For the subcategory $\mathfrak H^+$ of oriented hypermaps, where the underlying surface is oriented and without boundary, the parent group is the even subgroup
\[\Gamma=\Gamma_{\mathfrak H^+}=\Delta(\infty, \infty, \infty)\cong C_{\infty}*C_{\infty}\cong F_2\]
of index $2$ in $\Delta[\infty,2,\infty]$. This acts on the edges of the bipartite map, with $X$ and $Y$ using the local orientation to rotate them around their incident black and white vertices, so that $Z$ rotates them around incident faces. Again $\Delta(p,q,r)$ is the parent group for the subcategory of oriented hypermaps of type $(p,q,r)$. Hypermaps of type $(p,2,r)$ can be regarded as maps of type $\{r,p\}$ by deleting their white vertices; conversely maps correspond to hypermaps with $q=2$ (see Figure~\ref{Minfty}).
 
\smallskip

\noindent{\bf 5.} One can regard the category $\mathfrak D$ of dessins d'enfants, introduced by Grothendieck~\cite{Gro}, as the subcategory of $\mathfrak H^+$ consisting of its finite objects, where the bipartite graph is finite and the surface is compact. The parent group $\Gamma=\Delta(\infty, \infty, \infty)\cong F_2$ and its action are the same as for $\mathfrak H^+$.

\smallskip

Here we briefly mention two other classes of permutational categories where Theorem~\ref{isothm} applies.

\smallskip

\noindent{\bf 6.}  Abstract polytopes~\cite{MS} are higher-dimensional generalisations of maps. Those $n$-polytopes associated with the Schl\"afli symbol $\{p_1, \ldots, p_{n-1}\}$ can be regarded as transitive permutation representations of the string Coxeter group $\Gamma$ with presentation
\[\langle R_0, \ldots, R_n\mid R_i^2=(R_{i-1}R_i)^{p_i}=(R_iR_j)^2=1\;(|i-j|>1)\rangle,\]
acting on flags. For instance maps, in Example~1, correspond to the symbol $\{\infty,\infty\}$. However, in higher dimensions, not all transitive representations of $\Gamma$ correspond to abstract polytopes, since they need to satisfy the intersection property~\cite[Proposition~2B10]{MS}.

\smallskip

\noindent{\bf 7.}  Under suitable connectedness conditions (see~\cite{Mas, Mun} for example), the connected, unbranched coverings $Y\to X$ of a topological space $X$ can be identified with the transitive permutation representations $\theta:\Gamma\to S={\rm Sym}(\Omega)$ of its fundamental group $\Gamma=\pi_1X$, acting by unique path-lifting on the fibre $\Omega$ over a base-point in $X$. The automorphism group of an object $Y\to X$ in this category is its group of covering transformations, the centraliser in $S$ of the monodromy group $\theta(\Gamma)$ of the covering. For instance, dessins (see Example~5 above) correspond to finite unbranched coverings of the thrice-punctured sphere $X={\mathbb P}^1({\mathbb C})\setminus\{0, 1, \infty\}={\mathbb C}\setminus\{0,1\}$, and hence to transitive finite permutation representations of its fundamental group $\Gamma=\pi_1X\cong F_2\cong\Delta(\infty,\infty,\infty)$. If we compactify surfaces by filling in punctures, then the unit interval $[0,1]\subset {\mathbb P}^1({\mathbb C)}$ lifts to a bipartite map on the covering surface $Y$, with black and white vertices over $0$ and $1$, and face-centres over $\infty$. See~\cite{GG, JW, LZ} for further details of these and other properties of dessins.

%A map can be regarded as a hypermap $\mathcal H$ in which each vertex of one colour, say black, has valency dividing $2$. Such vertices can therefore be omitted, leaving an uncoloured map $\mathcal M$: omitting a vertex of valency $2$ yields an edge of $\mathcal M$, supporting two directed edges corresponding to the edges of $\mathcal H$ forming a $2$-cycle of $y$, while omitting a vertex of valency $1$ yields a half-edge of $\mathcal M$ corresponding to an edge of $\mathcal H$ fixed by $y$. Since $y^2=1$ in all cases we therefore impose the extra relation $Y^2=1$, or equivalently $(R_0R_2)^2=1$ in the unoriented case, and regard the new group $\Gamma$

%%%%%%%%%%%%%%%%%

\section{Preliminary results}\label{prelim}

In this section we will prove some general results which ensure that certain groups have various automorphism realisation properties.

\begin{lemm}\label{epi}
Let $\theta:\Gamma\to\Gamma'$ be an epimorphism of groups. If $\Gamma'$ has the finite, countable or strong automorphism realisation property, then so does $\Gamma$.
\end{lemm}

\begin{proof}
If $A\cong N'/M'$ where $M'\le N'\le \Gamma'$ and $N'=N_{\Gamma'}(M')$, then $A\cong N/M$ where $M=\theta^{-1}(M')$ and $N=\theta^{-1}(N')=N_{\Gamma}(M)$, with $|\Gamma:M|=|\Gamma':M'|$, so $\Gamma$ inherits the properties from $\Gamma'$.
\end{proof}

Note that non-conjugate subgroups $M'$ lift to non-conjugate subgroups $M$, so $A$ is realised by least as many conjugacy classes of subgroups in $\Gamma$ as in $\Gamma'$. Our basic tool for proving that a group $\Gamma$ has the FARP will be the following:

\begin{prop}\label{FARPprop}
Let $\Gamma$ be a group with a sequence $\{N_n\mid n\ge n_0\}$ of maximal subgroups $N_n$ of finite index such that for each $a, d\in{\mathbb N}$ there is some $n$ with $|N_n:K_n|>a$, where $K_n$ is the core of $N_n$ in $\Gamma$, and there is an epimorphism $N_n\to F_d$. Then $\Gamma$ has the finite automorphism realisation property. 
\end{prop}

\begin{proof}
Any finite group $A$ is an $d$-generator group for some $d\in\mathbb N$, so there is an epimorphism $F_d\to A$. By hypothesis, for some maximal subgroup $N=N_n$ of $\Gamma$ there is an epimorphism $N\to F_d$, and the core $K$ of $N$ satisfies $|N:K|>|A|$. Composition gives an epimorphism $N\to A$, and hence a normal subgroup $M$ of $N$ with $N/M\cong A$. Then $N_{\Gamma}(M)\ge N$, so the maximality of $N$ implies that either $N=N_{\Gamma}(M)$ or $M$ is a normal subgroup of $\Gamma$. If $M$ is normal in $\Gamma$ then $M$ must be contained in the core $K$ of $N$, so that $|N:M|\ge |N:K|$. But this is impossible, since $|N:M|=|A|$ and we chose $N=N_n$ so that $|N:K|>|A|$. Hence  $N=N_{\Gamma}(M)$, as required.
\end{proof}

\begin{rema}\label{FARPremark}
We can find such subgroups $N$ with $|N:K|$ arbitrarily large, so infinitely many of them are mutually non-conjugate, and hence so are their corresponding subgroups $M$, since conjugate subgroups have conjugate normalisers. Thus $A$ is realised by $\aleph_0$ non-isomorphic objects $\mathcal O\in\mathfrak C$.
\end{rema}

In order to deal with the CARP we need an analogue of Proposition~\ref{FARPprop} for countably infinite groups $A$. Here we have the advantage that, instead of an infinite sequence of maximal subgroups, which are finitely generated if $\Gamma$ is, a single infinitely generated maximal subgroup is sufficient. However, when $A$ is infinite we cannot ensure that $M$ is not normal in $\Gamma$ simply by comparing indices of subgroups, since these are no longer finite; a different idea is therefore needed.

\begin{prop}\label{CARP}
Let $\Gamma$ be a group with a non-normal maximal subgroup $N$ and an epimorphism $\phi:N\to F_{\infty}$. Then $\Gamma$ has the countable automorphism realisation property.
\end{prop}

\begin{proof}
Given any countable group $A$ there exist epimorphisms $\alpha:F_{\infty}\to A$; composing any of these with the epimorphism $\phi:N\to F_{\infty}$ gives an epimorphism $\phi\circ\alpha:N\to A$, and hence a normal subgroup $M=\ker(\phi\circ\alpha)$ of $N$ with $N/M\cong A$. As before, the maximality of $N$ implies that either $N=N_{\Gamma}(M)$, as required, or $M$ is a normal subgroup of $\Gamma$. In the latter case $M$ is contained in the core $K$ of $N$ in $\Gamma$, so to prove the result we need to show that we can choose $\alpha$ so that $M\not\le K$. Since $N$ is not normal in $\Gamma$ we have $N\setminus K\ne\emptyset$, so choose any element $g\in N\setminus K$, and define $f:=g\phi\in F_{\infty}$. Then we can choose $\alpha:F_{\infty}\to A$ so that all of the (finitely many) free generators of $F_{\infty}$ appearing in $f$ are in $\ker(\alpha)$, and hence $g\in M$. Thus $M\not\le K$, as required.
\end{proof}

\begin{rema}\label{CARPremark}
In fact, if $A\ne 1$ we can choose such epimorphisms $\alpha$ with $2^{\aleph_0}$ different kernels, lifting back to distinct subgroups $M$ of $\Gamma$; these all have normaliser $N$, which is its own normaliser in $\Gamma$, so they are mutually non-conjugate, giving us $2^{\aleph_0}$ non-isomorphic objects realising $A$. However, if $A=1$ we obtain only one object, corresponding to $M=N$.
\end{rema}

\begin{theo}\label{freeSARP}
If  a group $\Gamma$ has an epimorphism onto a non-abelian free group then it has the strong automorphism realisation property.
\end{theo}

\begin{proof}
By Lemma~\ref{epi} it is sufficient to prove this for $\Gamma=F_2=\langle X, Y\mid -\rangle$, since this is an epimorphic image of all other non-abelian free groups. We first prove that $\Gamma$ satisfies the FARP by showing that it satisfies the hypotheses of Proposition~\ref{FARPprop}.

For each $n\ge 3$ let $\theta_n:\Gamma\to S_n$ be the epimorphism defined by
\[X\mapsto x=(1, 2, \ldots, n),\quad Y\mapsto y=(1,2),\] 
and let $N_n$ be the subgroup of $\Gamma$ fixing $1$. Since $S_n$ is transitive on $\{1, 2,\ldots,n\}$ we have $|\Gamma:N_n|=n$, and since $S_n$ is primitive $N_n$ is a maximal subgroup of $\Gamma$. If $K_n$ is the core of $N_n$ in $\Gamma$ then $N_n/K_n\cong S_{n-1}$, so $|N_n:K_n|=(n-1)!$. By a theorem of Schreier (see~\cite[Ch.~I, Prop.~3.9]{LS}), $N_n\cong F_{n+1}$. Thus the hypotheses of Proposition~\ref{FARPprop} are satisfied, so $\Gamma$ satisfies the FARP.

We now use Proposition~\ref{CARP} to show that $\Gamma$ satisfies the CARP. Let us define a representation $\theta: \Gamma\to P$ of $\Gamma$ as a transitive permutation group $P$ on $\mathbb Z$, where $X$ acts as the translation $x: i\mapsto i+1$, and $Y$ induces the transposition $y=(0,1)$. Since $y_k:=y^{x^k}$ is the transposition $(k,k+1)$, $P$ contains the group of all permutations of $\mathbb Z$ with finite support. It follows that $P$ is doubly transitive, and hence primitive, so the subgroup $N_{\infty}$ of $\Gamma$ fixing $0$ is a non-normal maximal subgroup.

Since $\langle X\rangle$ acts regularly on $\mathbb Z$, its elements form a Schreier system of coset representatives for $N_{\infty}$ in $\Gamma$. The Reidemeister--Schreier algorithm~\cite[\S II.4]{LS} then shows that $N_{\infty}$ has generators $Y_k:=Y^{X^k}$ for $k\in{\mathbb Z}\setminus\{0, -1\}$, together with $XY^{\pm 1}$, and that there are no defining relations between them. Thus $N_{\infty}\cong F_{\infty}$, so the hypotheses of Proposition~\ref{CARP} are satisfied and the result follows.
\end{proof}

\begin{rema}
Although this paper is mainly about categories of maps and hypermaps, we note here that Theorem~\ref{freeSARP} applies to the fundamental groups $\Gamma$ of many topological spaces, so their categories of coverings satisfy the SARP. Examples include compact orientable surfaces of genus $g$ with $k$ punctures, where $2g+k\ge 3$. Taking $g=0$, $k=3$ shows that the category $\mathfrak D$ of all dessins satisfies the FARP (see Theorem~\ref{realisation}(a) for a more specific result); in fact, Cori and Mach\`\i\/~\cite{CoMa} proved that every finite group is the automorphism group of a finite oriented hypermap, two years before Grothendieck introduced dessins in~\cite{Gro}.

Hidalgo~\cite{Hid} has proved the stronger result that every action of a finite group $A$ by orientation-preserving self-homeomorphisms of a compact oriented surface $S$ is topologically equivalent to the automorphism group of a dessin. One way to see this is to triangulate $S/A$, with all critical values of the projection $\pi:S\to S/A$ among the vertices, none of which has valency $1$, and then to add an edge to an additional $1$-valent vertex $v$ in the interior of a face. This gives a map (that is, a dessin with $q=2$) on $S/A$ which lifts via $\pi$ to a dessin $\mathcal D$ on $S$ with $A\le{\rm Aut}({\mathcal D})$. The only $1$-valent vertices in $\mathcal D$ are the $|A|$ vertices in $\pi^{-1}(v)$. These are permuted by ${\rm Aut}({\mathcal D})$, with $A$ and hence ${\rm Aut}({\mathcal D})$ acting transitively; however, the stabiliser of a $1$-valent vertex (in any dessin) must be the identity, so ${\rm Aut}({\mathcal D})=A$, as required. By starting with inequivalent triangulations of $S/A$ one can obtain infinitely many non-isomorphic dessins realising this action of $A$.
\end{rema}

\begin{rema}\label{gpP}
The permutation group $P$ appearing in the proof of Theorem~\ref{freeSARP} has some interesting properties. The normal closure of $y$ in $P$ is the group $F:={\rm FSym}({\mathbb Z})$ of all finitary permutations of $\mathbb Z$ (those permutations with finite support). As an infinite but locally finite group, $F$ is not finitely generated. Regarding it as the union of an ascending tower of finite symmetric groups, and using their standard presentation~\cite[(6.22)]{CM}, we see that $F$ has generators $y_k=(k,k+1)$ for $k\in{\bf Z}$, with defining relations $(y_ky_l)^{m_{kl}}=1$ where $m_{kl}=1, 3$ or $2$ as $|k-l|=0, 1$ or otherwise, so it can be regarded as a Coxeter group of infinite rank. It is complemented in $P$ by the infinite cyclic group $\langle x\rangle$, which induces the automorphism $y_k\mapsto y_{k+1}$ of $F$ by conjugation. Thus $P$ has generators $x, y$ with defining relations $y^2=(yy^x)^3=(yy^{x^i})^2=1\;(i\ge 2)$. The elements of $P$ are those permutations of $\mathbb Z$ which are `almost translations', that is, which differ from a translation $x^i$ in only finitely many places. We will revisit $P$ in the final section of this paper, to consider the maps associated with it.
\end{rema}

%%%%%%%%%%%%%%%%%

\section{The FARP for hyperbolic triangle groups}\label{FARP}

\iffalse

This can be deduced from a more general and very deep result~\cite[Theorem~2.1]{BeLu} of Belolipetsky and Lubotzky, which implies the FARP for every finitely generated group $\Gamma$ which is `large' in the sense of having a normal subgroup of finite index with an epimorphism onto a non-abelian free group. Each cocompact hyperbolic triangle group $\Gamma$ is large since, being resldually finite with only finitely many conjugacy classes of torsion elements (represented by the powers of its standard elliptic generators), $\Gamma$ has a torsion-free normal subgroup of finite index, which must be a surface group of genus $g\ge 2$ and therefore can be mapped onto the free group $F_g$; the extension to non-cocompact groups and extended triangle groups is straightforward. However, the proof of~\cite[Theorem~2.1]{BeLu} is long, delicate and non-constructive, so here we offer a shorter, more direct argument, specific to the context of this paper.

 \begin{theo}\label{hypFARP}
Each triangle group $\Delta(p,q,r)$ or extended triangle group $\Delta[p,q,r]$ of hyperbolic type, together with its associated category of oriented or unoriented hypermaps, has the finite automorphism realisation property.
\end{theo}

\begin{proof}

\fi

In this section we prove Theorem~\ref{realisation}(a), starting with hyperbolic triangle groups $\Gamma=\Delta(p, q,r)$. First assume that $\Gamma$ (acting on the hyperbolic plane) is cocompact, that is, its periods $p, q$ and $r$ are all finite. By Dirichlet's Theorem on primes in an arithmetic progression, there are infinitely many primes $n\equiv -1$ mod $(l)$, where $l:={\rm lcm}\{2p,2q,2r\}$. For each such $n$ there is an epimorphism $\theta_n:\Gamma\to  {\rm PSL}_2(n)$ sending the standard generators $X, Y$ and $Z$ of $\Gamma$ to elements of ${\rm PSL}_2(n)$ of orders $p, q$ and $r$ (see~\cite[Corollary~C]{Gar}, for example). This gives an action of $\Gamma$ on the projective line ${\mathbb P}^1(\mathbb F_n)$, which is doubly transitive and hence primitive, so the subgroup $N_n$ of $\Gamma$ fixing $\infty$ is a non-normal maximal subgroup of index $n+1$. Since $p, q$ and $r$ all divide $(n+1)/2$, the generators $X, Y$ and $Z$ induce semi-regular permutations on ${\mathbb P}^1(\mathbb F_n)$, with all their cycles of length $p, q$ or $r$. Thus no non-identity powers of $X$, $Y$ or $Z$ have fixed points, so by a theorem of Singerman~\cite{Sin} $N_n$ is a surface group
\[N_n=\langle A_i, B_i\;(i=1,\ldots, g)\mid \prod_{i=1}^g[A_i,B_i]=1\rangle\]
of genus $g$ given by the Riemann--Hurwitz formula:
\begin{equation}\label{RHeqn}
2(g-1)=(n+1)\left(1-\frac{1}{p}-\frac{1}{q}-\frac{1}{r}\right).
\end{equation}
This shows that $g\to\infty$ as $n\to\infty$. Now we can map $N_n$ onto the free group $F_g$ by sending the generators $A_i$ to a free basis, and the generators $B_i$ to $1$. The core $K_n=\ker(\theta_n)$ of $N_n$ in $\Gamma$ satisfies $|N_n:K_n|=|{\rm PSL}_2(n)|/(n+1)=n(n-1)/2$, so Proposition~\ref{FARPprop} can be applied to give the result.

Now assume that $\Gamma$ is not cocompact, with $k$ infinite periods $p, q, r$ for some $k=1, 2$ or $3$. We can adapt the above argument by first choosing an infinite set of primes $n\ge 13$ such that any finite periods of $\Gamma$ divide $(n+1)/2$, as before. For each such $n$ we can map $\Gamma$ onto a cocompact triangle group $\Gamma_n$, where each infinite period of $\Gamma$  is replaced with $(n+1)/2$. Since $(n+1)/2\ge 7$, the triangle group $\Gamma_n$ is also hyperbolic, so as before there is an epimorphism $\Gamma_n\to {\rm PSL}_2(n)$, giving (by composition) a primitive action of $\Gamma$ on ${\mathbb P}^1(\mathbb F_n)$. Again, no non-identity powers of any elliptic generators among $X, Y$ and $Z$ have fixed points, but any parabolic generator (one of infinite order) now induces two cycles of length $(n+1)/2$ on ${\mathbb P}^1(\mathbb F_n)$, so by~\cite{Sin} it introduces two parabolic generators $P_i$ into the standard presentation of the point-stabiliser $N_n$ in $\Gamma$. We therefore have
\[N_n=\langle A_i, B_i\;(i=1,\ldots, g), P_i\;(i=1,\ldots, 2k)\mid \prod_{i=1}^g[A_i,B_i].\prod_{i=1}^{2k}P_i=1\rangle,\]
a free group of rank $2g+2k-1$, where the Riemann--Hurwitz formula now gives
\begin{equation}
2(g-1)+2k=(n+1)\left(1-\frac{1}{p}-\frac{1}{q}-\frac{1}{r}\right)
\end{equation}
with $1/\infty=0$. Since $k\le 3$ we have $g\to\infty$ as $n\to\infty$, so Proposition~\ref{FARPprop} again gives the result.

The proof when $\Gamma$ is an extended triangle group $\Delta[p,q,r]$ of hyperbolic type is similar to that for $\Delta(p,q,r)$. If $\Gamma$ is cocompact then, as before, we consider epimorphisms $\theta_n:\Gamma^+=\Delta(p,q,r)\to{\rm PSL}_2(n)$ for primes $n\equiv -1$ mod~$(l)$, where now $l={\rm lcm}\{2p, 2q, 2r, 4\}$; the stabilisers of $\infty$ form a series of maximal subgroups $N_n$ of index $n+1$ in $\Gamma^+$. By an observation of Singerman~\cite{Sin74} the core $K_n$ of $N_n$ in $\Gamma^+$ is normal in $\Gamma$, with quotient $\Gamma/K_n$ isomorphic to ${\rm PSL}_2(n)\times C_2$ or ${\rm PGL}_2(n)$ as the automorphism of ${\rm PSL}_2(n)$ inverting $x$ and $y$ is inner or not. Thus $\theta_n$ extends to a homomorphism $\theta^*_n:\Gamma\to {\rm PGL}_2(n)$; in the first case its image is ${\rm PSL}_2(n)$ and its kernel $K^*_n$ contains $K_n$ with index $2$, and in the second case it is an epimorphism with kernel $K^*_n=K_n$. In either case the action of $\Gamma^+$ on ${\mathbb P}^1(\mathbb F_n)$ extends to an action of $\Gamma$, and the stabiliser $N^*_n$ of $\infty$ is a maximal subgroup of index $n+1$ in $\Gamma$, containing $N_n$ with index $2$. In order to apply Proposition~\ref{FARPprop} to these subgroups $N^*_n$ it is sufficient to show that they map onto free groups of unbounded rank.

Now $N^*_n$ is a non-euclidean crystallographic (NEC) group, and $N_n$ is its canonical Fuchsian subgroup of index $2$. We can obtain the signature of $N^*_n$ by using Hoare's extension to NEC groups~\cite{Hoa} of Singerman's results~\cite{Sin}  on subgroups of Fuchsian groups. As before, $N_n$ is a surface group of genus $g$ given by (\ref{RHeqn}). There are no elliptic or parabolic elements in $N_n$, and hence none in $N^*_n$. The reflections $R_i\;(i=0,1,2)$ generating $\Gamma$ induce involutions on $\mathbb P^1(\mathbb F_n)$, each with at most two fixed points. If $\Gamma/K_n\cong {\rm PSL}_2(n)\times C_2$ these involutions are elements of ${\rm PSL}_2(n)$, so they are even permutations by the simplicity of this group, and hence they have no fixed points since $n+1\equiv 0$ mod~$(4)$. Thus $N^*_n$ contains no reflections; however, it is not a subgroup of $\Gamma^+$, so it is a non-orientable surface group
\[N_n^*=\langle G_1,\ldots, G_{g^*}\mid G_1^2\ldots G_{g^*}^2=1\rangle\]
generated by glide-reflections $G_i$, with its genus $g^*$ given by the Riemann--Hurwitz formula
\[2-2g=2(2-g^*)\]
for the inclusion $N_n\le N_n^*$, so $g^*=g+1$. Thus there is an epimorphism $N^*_n\to F_d=\langle X_1,\ldots, X_d\mid -\rangle$ where $d=\lfloor g^*/2\rfloor$, given by $G_{2i-1}\mapsto X_i$ and $G_{2i}\mapsto X_i^{-1}$ for $i=1,\ldots, d$ and $G_{g^*}\mapsto 1$ if $g^*$ is odd. Since $g\to\infty$ as $n\to\infty$, we have $d\sim g/2\to\infty$ also, so Proposition~\ref{FARPprop} gives the result.

Similar arguments also deal with the case where $\Gamma/K_n\cong{\rm PGL}_2(n)$. Since $n\equiv -1$ mod~$(4)$, each generating reflection $R_i$ of $\Gamma$ induces an odd permutation of $\mathbb P^1(\mathbb F_n)$ with two fixed points, contributing two reflections to the standard presentation of the NEC group $N^*_n$. The Riemann--Hurwitz formula for the inclusion $N_n\le N^*_n$ then takes the form
\[2-2g=2(2-h^*+s),\]
where $h^*=2g^*$ or $g^*$ as $N^*_n$ has an orientable or non-orientable quotient surface of genus $g^*$ with $s$ boundary components for some $s\le 6$. Thus $h^*\sim g\to\infty$ as $n\to\infty$. We obtain an epimorphism $N^*_n\to F_d$ with $d\sim h^*/2$ as in the orientable or non-orientable cases above, this time by mapping the additional standard generators of $N^*_n$, associated with the boundary components, to $1$, so Proposition~\ref{FARPprop} again gives the result. Finally, in the non-cocompact case, any periods $p, q, r=\infty$ can be dealt with as above for $\Delta(p,q,r)$. This completes the proof of Theorem~\ref{realisation}(a).
%\end{proof}

\begin{rema}
The restrictions on the prime $n$ in the above proof are mainly for convenience of exposition, rather than necessity. Relaxing them would allow $X, Y$ and $Z$ to have one or two fixed points on $\mathbb P^1(\mathbb F_n)$, thus adding extra standard generators to $N_n$ and $N^*_n$ and extra summands to the Riemann--Hurwitz formulae used. However, these extra terms are bounded as $n\to\infty$, so asymptotically they make no significant difference. It seems plausible that an argument based on the \v Cebotarev Density Theorem would show that, given $\Gamma=\Delta[p,q,r]$, the cases $\Gamma/K_n\cong {\rm PSL}_2(n)\times C_2$ and ${\rm PGL}_2(n)$ each occur for infinitely many primes $n\equiv-1$ mod~$(l)$, so that only one case would need to be considered; however, the resulting shortening of the proof would not justify the effort.
\end{rema}

%%%%%%%%%%%%%%%%%%%

\section{The CARP for hyperbolic triangle groups}

%We now consider the CARP for hyperbolic triangle groups:
\iffalse
\begin{prob}\label{question}
Does each hyperbolic triangle group (and hence the category of oriented hypermaps of that type) have the countable automorphism realisation property?
\end{prob}
\fi

In this section we consider Theorem~\ref{realisation}(b) for hyperbolic triangle groups $\Gamma=\Delta(p,q,r)$. Having proved that they have all the FARP, it is sufficient to consider the CARP. We need to show that $\Gamma$ satisfies the hypotheses of Proposition~\ref{CARP}, that is, it has a non-normal maximal subgroup which has an epimorphism onto $F_{\infty}$.
Given $\Gamma$, it is easy to find maximal subgroups of finite index by mapping $\Gamma$ onto primitive permutation groups of finite degree; however, such subgroups are finitely generated, so they do not map onto $F_{\infty}$; a maximal subgroup of infinite index is needed, and these seem to be harder to find. They certainly exist: by a result of Ol'shanski\u\i\/~\cite{Ols00}, $\Gamma$ has a quotient $Q\ne 1$ with no proper subgroups of finite index; by Zorn's Lemma, $Q$ has maximal subgroups, which must have infinite index, and these lift back to maximal subgroups of infinite index in $\Gamma$. These are not normal (otherwise they would have prime index), but does one of them map onto $F_{\infty}$? Conceivably, they could be generated by elliptic elements, which have finite order, in which case they would not map onto a free group of any rank. As a first step we consider the case where $\Gamma$ is not cocompact, that is, it has an infinite period, so it is a free product of two cyclic groups.

\begin{theo}\label{CpCq}
If $\Gamma$ is a hyperbolic triangle group $\Delta(p,q,r)$ with at least one infinite period, then $\Gamma$ and the corresponding category of oriented hypermaps have the countable automorphism realisation property.
\end{theo}

\begin{proof}
By Proposition~\ref{CARP} it is sufficient  to show that $\Gamma$ has a non-normal maximal subgroup $N$ with an epimorphism $N\to F_{\infty}$.  Using the usual isomorphisms between triangle groups, we may assume that $r=\infty$, and that $p\ge 3$ and $q\ge 2$, so that $\Gamma\cong C_p*C_q$ with $p, q\in{\mathbb N}\cup\{\infty\}$.

\medskip

\noindent{\sl Case 1: $p=3, q=2$.} First we consider the case where $p=3$ and $q=2$, so that $\Gamma$ is isomorphic to the modular group ${\rm PSL}_2({\mathbb Z})\cong C_3*C_2$. We can construct a maximal subgroup $N$ of infinite index in $\Gamma$ as the point stabiliser in a primitive permutation representation of $\Gamma$ of infinite degree. Since $\Gamma$ is the parent group
\[\Delta(3,2,\infty)=\langle X, Y, Z\mid X^3=Y^2=XYZ=1\rangle\]
for the category ${\mathfrak C}={\mathfrak M}_3^+$ of oriented trivalent maps, we can take $N$ to be the subgroup of $\Gamma$ corresponding to an infinite map $\mathcal N_3$ in $\mathfrak C$. 

 \iffalse First we consider the case where $p=3$ and $q=2$, so that $\Gamma$ is isomorphic to the modular group ${\rm PSL}_2({\mathbb Z})\cong C_3*C_2$. Since $\Gamma$ is the parent group
\[\Delta(3,2,\infty)=\langle X, Y, Z\mid X^3=Y^2=XYZ=1\rangle\]
for the category ${\mathfrak C}={\mathfrak M}_3^+$ of oriented trivalent maps, we can construct such a maximal subgroup $N$ of infinite index in $\Gamma$ as the subgroup corresponding to a map $\mathcal N_3$ in $\mathfrak C$, or equivalently as the point stabiliser in a permutation representation of $\Gamma$, constructed to be primitive and of infinite degree. \fi

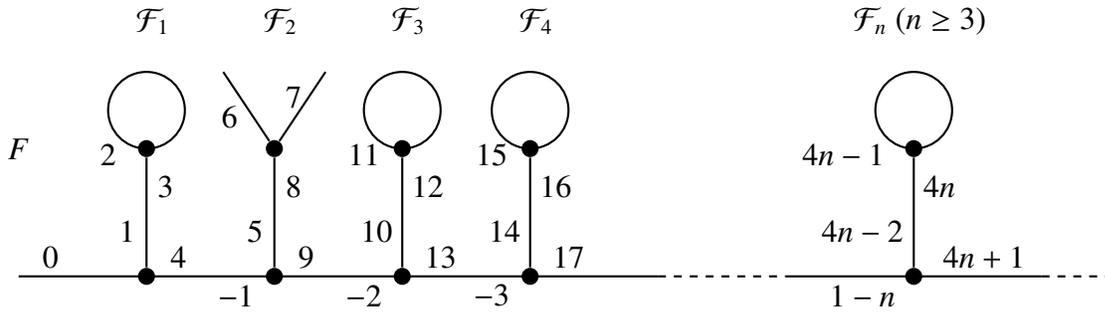
\begin{figure}[h!]
\begin{center}
\begin{tikzpicture}[scale=0.17, inner sep=0.8mm]

%\node (-2) at (-20,0) [shape=circle, fill=black] {};
%\node (-1) at (-10,0) [shape=circle, fill=black] {};
%\node (0) at (0,0) [shape=circle, fill=black] {};
\node (1) at (10,0) [shape=circle, fill=black] {};
\node (2) at (20,0) [shape=circle, fill=black] {};
\node (3) at (30,0) [shape=circle, fill=black] {};
\node (4) at (40,0) [shape=circle, fill=black] {};
%\node (5) at (50,0) [shape=circle, fill=black] {};
%\node (6) at (60,0) [shape=circle, fill=black] {};
\node (7) at (70,0) [shape=circle, fill=black] {};
%\node (8) at (80,0) [shape=circle, fill=black] {};
%\node (9) at (90,0) [shape=circle, fill=black] {};

%\node (-2-) at (-20,-10) [shape=circle, fill=black] {};
%\node (-1-) at (-10,-10) [shape=circle, fill=black] {};
\node (1+) at (10,10) [shape=circle, fill=black] {};
\node (2+) at (20,10) [shape=circle, fill=black] {};
\node (3+) at (30,10) [shape=circle, fill=black] {};
\node (4+) at (40,10) [shape=circle, fill=black] {};
%\node (5+) at (50,10) [shape=circle, fill=black] {};
%\node (6+) at (60,10) [shape=circle, fill=black] {};
\node (7+) at (70,10) [shape=circle, fill=black] {};
%\node (8-) at (80,-10) [shape=circle, fill=black] {};
%\node (9+) at (90,10) [shape=circle, fill=black] {};

\draw [thick] (0,0) to (50,0);
%\draw [thick,dashed] (-25,0) to (-31,0);
\draw [thick,dashed] (50,0) to (60,0);
\draw [thick] (60,0) to (80,0);
\draw [thick,dashed] (80,0) to (86,0);
%\draw [thick] (-2) to (-2-);
%\draw [thick] (-1) to (-1-);
%\draw [thick] (0,-10) to (0,10);
\draw [thick] (1) to (1+);
\draw [thick] (2) to (2+);
\draw [thick] (3) to (3+);
\draw [thick] (4) to (4+);
%\draw [thick] (5) to (5-);
%\draw [thick] (6) to (6+);
\draw [thick] (7) to (7+);
%\draw [thick] (8) to (8-);
%\draw [thick] (9) to (9+);

\draw [thick] (13,13) arc (0:360:3);
\draw [thick] (33,13) arc (0:360:3);
\draw [thick] (43,13) arc (0:360:3);
%\draw [thick] (65,15) arc (0:360:5);
\draw [thick] (73,13) arc (0:360:3);

\draw [thick] (16,16) to (2+) to (24,16);

\node at (37,-1.5) {$-3$};
\node at (27,-1.5) {$-2$};
\node at (17,-1.5) {$-1$};
\node at (2.5,1.5) {$0$};
\node at (8.5,3.5) {$1$};
\node at (7,9.5) {$2$};
\node at (11.5,7) {$3$};
\node at (12.5,1.5) {$4$};
\node at (18.5,3.5) {$5$};
\node at (16.5,12.5) {$6$};
\node at (21.5,14) {$7$};
\node at (21.5,7) {$8$};
\node at (22.5,1.5) {$9$};
\node at (28,3.5) {$10$};
\node at (27,9.5) {$11$};
\node at (32,7) {$12$};
\node at (33,1.5) {$13$};
\node at (38,3.5) {$14$};
\node at (37,9.5) {$15$};
\node at (42,7) {$16$};
\node at (43,1.5) {$17$};
\node at (66,3.5) {$4n-2$};
\node at (64.5,9.5) {$4n-1$};
\node at (72,7) {$4n$};
\node at (75.5,1.5) {$4n+1$};
\node at (66,-1.5) {$1-n$};

\node at (0,10) {$F$};
\node at (10.5,20) {${\mathcal F}_1$};
\node at (20.5,20) {${\mathcal F}_2$};
\node at (30.5,20) {${\mathcal F}_3$};
\node at (40.5,20) {${\mathcal F}_4$};
\node at (70.5,20) {${\mathcal F}_n\;(n\ge 3)$};

\end{tikzpicture}
\end{center}
\caption{The trivalent map ${\mathcal N}_3$} 
\label{mapN3}
\end{figure}

We will take $\mathcal N_3$ to be the infinite planar trivalent map shown in Figure~\ref{mapN3}, oriented with the positive (anticlockwise) orientation of the plane. The monodromy group $G=\langle x, y\rangle$ of this map gives a transitive permutation representation $\theta:\Gamma\to G, X\mapsto x, Y\mapsto y, Z\mapsto z$ of $\Gamma$ on the set $\Omega$ of directed edges of $\mathcal N_3$, with $x$ rotating them anticlockwise around their target vertices, and $y$ reversing their direction. The vertices, all of valency $3$, correspond to the $3$-cycles of $x$ (it has no fixed points). The edges correspond to the cycles of $y$, with three free edges corresponding to its fixed points and the other edges corresponding to its $2$-cycles. The faces correspond to the cycles of $z=yx^{-1}$, and in particular, the directed edge $\alpha$ labelled $0$ and fixed by $y$ is in an infinite cycle $C=(\ldots, \alpha z^{-1}, \alpha, \alpha z, \ldots)$ of $z$, corresponding to the unbounded face $F$ of $\mathcal N_3$; the directed edges $\alpha z^i$ in $C$ are indicated by integers $i$ in Figure~\ref{mapN3}. The unlabelled directed edges are fixed points of $z$, one incident with each $1$-valent face. The pattern seen in Figure~\ref{mapN3} repeats to the right in the obvious way. The `flowers' ${\mathcal F}_n\;(n\ge 1)$ above the horizontal axis continue indefinitely to the right, with ${\mathcal F}_n$ an identical copy of ${\mathcal F}_1$ for each $n\ge 3$; we will later need the fact that for each $n\ge 2$ the `stem' of ${\mathcal F}_n$ (the vertical edge connecting it to the horizontal axis) carries two directed edges in $C$, with only one of their two labels divisible by $n$.

\begin{lemm}\label{primitive}
The group $\Gamma$ acts primitively on $\Omega$.
\end{lemm}

\begin{proof}
Suppose that $\sim$ is a $\Gamma$-invariant equivalence relation on $\Omega$; we need to show that it is either the identity or the universal relation. Since $\alpha y=\alpha$, the equivalence class $E=[\alpha]$ containing $\alpha$ satisfies $Ey=E$. Since $\langle Z\rangle$ acts regularly on $C$ we can identify $C$ with $\mathbb Z$ by identifying each $\alpha z^i\in C$ with the integer $i$, so that $Z$ acts by $i\mapsto i+1$. Then $\sim$ restricts to a translation-invariant equivalence relation on $\mathbb Z$, which must be congruence mod~$(n)$ for some $n\in{\mathbb N}\cup\{\infty\}$, where we include $n=1$ and $\infty$ for the universal and  identity relations on $\mathbb Z$.

Suppose first that $n\in\mathbb N$, so $E\cap C$ is the subgroup $(n)$ of $\mathbb Z$. If $n=1$ then $C\subseteq E$. Now $Ex^{-1}$ is an equivalence class, and it contains $\alpha x^{-1}=1$; this is in $C$, and hence in $E$, so $Ex^{-1}=E$. We have seen that $Ey=E$, so $E=\Omega$ since $G=\langle x^{-1}, y\rangle$, and hence $\sim$ is the universal relation on $\Omega$.

We may therefore assume that $n>1$. The vertical stem of the flower ${\mathcal F}_n$ is an edge carrying two directed edges in $C$, with only one of its two labels divisible by $n$, so one is in $E$ whereas the other is not. However, these two directed edges are transposed by $y$, contradicting the fact that $Ey=E$.

Finally suppose that $n=\infty$, so that all elements of $C$ are in distinct equivalence classes, and hence the same applies to $Cy$. In particular, since $\alpha\in C\cap Cy$ we have $E\cap C=\{\alpha\}=E\cap Cy$. By inspection of Figure~\ref{mapN3}, $\Omega=C\cup Cy$ and hence $E=\{\alpha\}$. It follows that all equivalence classes for $\sim$ are singletons, so $\sim$ is the identity relation, as required.
\end{proof}

We now return to the proof of Theorem~\ref{CpCq}. It follows from Lemma~\ref{primitive} that the subgroup $N=\Gamma_{\alpha}$ of $\Gamma$ fixing $\alpha$ is maximal. Clearly $N$ is not normal in $\Gamma$, since $G$ is not a regular permutation group, so it sufficient to find an epimorphism $N\to F_{\infty}$. One could use the Reidemeister--Schreier algorithm to find a presentation for $N$: truncation converts $\mathcal N_3$ into a coset diagram for $N$ in $\Gamma$, and then deleting edges to form a spanning tree yields a Schreier transversal. In fact a glance at Figure~\ref{mapN3} shows that $N$ is a free product of cyclic groups: three of these, corresponding to the fixed points of $y$ and generated by conjugates of $Y$, have order $2$, and there are infinitely many of infinite order, generated by conjugates of $Z$ and corresponding to the fixed points of $z$, that is, the $1$-valent faces of $\mathcal N_3$, one in each flower ${\mathcal F}_n$ for $n\ne 2$. By mapping the generators of finite order to the identity we obtain the required epimorphism $N\to F_{\infty}$.  

\medskip

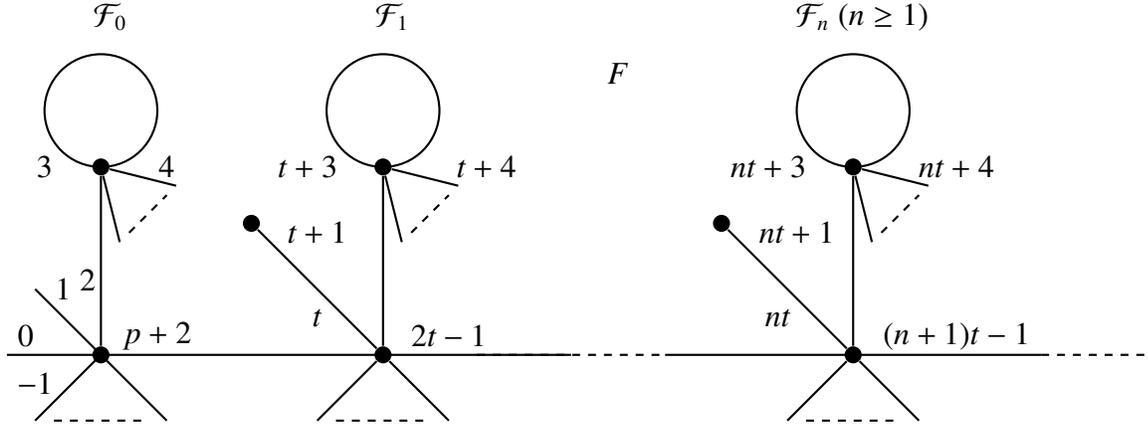
\begin{figure}[h!]
\begin{center}
\begin{tikzpicture}[scale=0.25, inner sep=0.8mm]

\node (3) at (30,0) [shape=circle, fill=black] {};
\node (4) at (45,0) [shape=circle, fill=black] {};
\node (4') at (38,7) [shape=circle, fill=black] {};
\node (7) at (70,0) [shape=circle, fill=black] {};
\node (7') at (63,7) [shape=circle, fill=black] {};
\node (3+) at (30,10) [shape=circle, fill=black] {};
\node (4+) at (45,10) [shape=circle, fill=black] {};
\node (7+) at (70,10) [shape=circle, fill=black] {};

\draw [thick] (25,0) to (55,0);
\draw [thick] (30,0) to (26.5,3.5);
\draw [thick] (30,0) to (26.5,-3.5);
\draw [thick] (30,0) to (33.5,-3.5);
\draw [thick, dashed] (27.5,-3.5) to (32.5,-3.5);
\draw [thick] (3+) to (31,6);
\draw [thick] (3+) to (34,9);
\draw [thick, dashed] (31.5,6.5) to (33.5,8.5);
\draw [thick] (4) to (4');
\draw [thick] (4) to (41.5,-3.5);
\draw [thick] (4) to (48.5,-3.5);
\draw [thick, dashed] (42.5, -3.5) to (47.5,-3.5);
\draw [thick] (4+) to (46,6);
\draw [thick] (4+) to (49,9);
\draw [thick, dashed] (46.5,6.5) to (48.5,8.5);
\draw [thick] (7) to (7');
\draw [thick] (7) to (66.5,-3.5);
\draw [thick] (7) to (73.5,-3.5);
\draw [thick] (7+) to (71,6);
\draw [thick] (7+) to (74,9);
\draw [thick, dashed] (71.5,6.5) to (73.5,8.5);
\draw [thick, dashed] (67.5, -3.5) to (72.5,-3.5);
\draw [thick, dashed] (50,0) to (60,0);
\draw [thick] (60,0) to (80,0);
\draw [thick,dashed] (80,0) to (86,0);

\draw [thick] (3) to (3+);
\draw [thick] (4) to (4+);

\draw [thick] (7) to (7+);

\draw [thick] (33,13) arc (0:360:3);
\draw [thick] (48,13) arc (0:360:3);
\draw [thick] (73,13) arc (0:360:3);

\node at (26.5,-1.5) {$-1$};
\node at (26,1) {$0$};
\node at (28,3.5) {$1$};
\node at (29.3,4) {$2$};
\node at (27,10) {$3$};
\node at (33.5,10) {$4$};
\node at (33,1) {$p+2$};
\node at (41.5,2) {$t$};
\node at (41.5,6.5) {$t+1$};
\node at (41,10) {$t+3$};
\node at (50.5,10) {$t+4$};
\node at (48.5,1) {$2t-1$};
\node at (66,2) {$nt$};
\node at (67,6.5) {$nt+1$};
\node at (65.5,10) {$nt+3$};
\node at (75.5,10) {$nt+4$};
\node at (75.5,1) {$(n+1)t-1$};

\node at (57.5,15) {$F$};
\node at (30.5,18) {${\mathcal F}_0$};
\node at (45.5,18) {${\mathcal F}_1$};
\node at (70.5,18) {${\mathcal F}_n\;(n\ge 1)$};

\end{tikzpicture}
\end{center}
\caption{The $p$-valent map ${\mathcal N}_p$, with $t:=p+3$} 
\label{mapNp}
\end{figure}

\noindent{\sl Case 2: Finite $p\ge 4, q=2$.} 
We now assume that $\Gamma$ is a Hecke group $C_p*C_2$ for some finite $p\ge 4$. Let ${\mathcal N}_p$ be the infinite $p$-valent planar map in Figure~\ref{mapNp}. Apart from ${\mathcal F}_0$, the flowers are all identical copies of ${\mathcal F}_1$, with a `leaf' growing out of its base and leading to a vertex of valency $1$, representing a fixed point of $x$. The `fans' indicated by short dashed lines represent however many free edges are needed in order that the incident vertex should have valency $p$, that is, $p-3$ free edges for vertices at the top of a stem, and $p-4$ for those at the base. As before, the elements $\alpha z^i$ of the cycle $C$ of $z$ containing the directed edge $\alpha=0$ are labelled with integers $i$;  to save space in the diagram only a few labels are shown. We define $t=p+3$, since a translation from a flower ${\mathcal F}_n\;(n\ge 1)$ to the next flower ${\mathcal F}_{n+1}$ adds that number to all labels.

The proof that the monodromy group $G=\langle x, y\rangle$ of ${\mathcal N}_p$ is primitive is very similar to that in Lemma~\ref{primitive} for $\mathcal N_3$. Any $\Gamma$-invariant equivalence relation $\sim$ on $\Omega$ restricts to $C$ as congruence mod~$(n)$ for some $n\in{\mathbb N}\cup\{\infty\}$. The equivalence class $E=[\alpha]$ satisfies $Ey=E$, so if $n\ne 1, \infty$ then the fact that $y$ transposes the directed edges labelled $nt$ and $nt+1$, with the first but not the second in $E=(n)$, gives a contradiction. If $n=1$ then $C\subseteq E$, so both $x$ and $y$ preserve $E$ and hence $\sim$ is the universal relation. If $n=\infty$ then $E\cap C=\{\alpha\}$, and hence also $E\cap Cy=\{\alpha\}$; but $\Omega=C\cup Cy$ and hence $E=\{\alpha\}$ and $\sim$ is the identity relation.

This shows that the subgroup $N$ of $\Gamma$ fixing $\alpha$ is maximal. As before, it is not normal, and it is a free product of cyclic groups, now of order $p$, $2$ or $\infty$, corresponding to the fixed points of $x, y$ and $z$ (infinitely many in each case). Sending the generators of finite order to the identity gives the required epimorphism $N\to F_{\infty}$.

\medskip

\noindent{\sl Case 3: Finite $p, q\ge 3$.} We modify the map ${\mathcal N}_p$ in the proof of Case~2 by removing the leaf attached to the base of each flower ${\mathcal F}_n\;(n\ge 1)$, adding a white vertex to every remaining edge (including one at the free end of each free edge), and finally adding edges incident with $1$-valent black vertices where necessary to ensure that all white vertices have valency $q$ or $1$. The resulting map ${\mathcal N}_{p,q}$ is shown in Figure~\ref{mapNpq}. 

\begin{figure}[h!]
\begin{center}
\begin{tikzpicture}[scale=0.28, inner sep=0.8mm]

\node (3) at (30,0) [shape=circle, fill=black] {};
\node (3+) at (30,10) [shape=circle, fill=black] {};
\node (3') at (37.5,4) [shape=circle, fill=black] {};
\node (3l) at (26,0) [shape=circle, draw] {};
\node (3ll) at (28,-3) [shape=circle, draw] {};
\node (3rl) at (32,-3) [shape=circle, draw] {};
\node (3+l) at (28,19) [shape=circle, fill=black] {};
\node (3+r) at (32,19) [shape=circle, fill=black] {};
\node (3w) at (30,5) [shape=circle, draw] {};
\node (3+w) at (30,16) [shape=circle, draw] {};
\node (3'w) at (37.5,0) [shape=circle, draw]{};
\node (3'l) at (35.5,-3) [shape=circle, fill=black] {};
\node (3'r) at (39.5,-3) [shape=circle, fill=black] {};
\node (3wu) at (33,6.5) [shape=circle, fill=black] {};
\node (3wl) at (33,3.5) [shape=circle, fill=black] {};

\node (4) at (45,0) [shape=circle, fill=black] {};
\node (4+) at (45,10) [shape=circle, fill=black] {};
\node (4') at (52.5,4) [shape=circle, fill=black] {};
\node (4ll) at (43,-3) [shape=circle, draw] {};
\node (4rl) at (47,-3) [shape=circle, draw] {};
\node (4+l) at (43,19) [shape=circle, fill=black] {};
\node (4+r) at (47,19) [shape=circle, fill=black] {};
\node (4w) at (45,5) [shape=circle, draw] {};
\node (4+w) at (45,16) [shape=circle, draw] {};
\node (4'w) at (52.5,0) [shape=circle, draw]{};
\node (4'l) at (50.5,-3) [shape=circle, fill=black] {};
\node (4'r) at (54.5,-3) [shape=circle, fill=black] {};
\node (4wu) at (48,6.5) [shape=circle, fill=black] {};
\node (4wl) at (48,3.5) [shape=circle, fill=black] {};

\node (7) at (70,0) [shape=circle, fill=black] {};
\node (7+) at (70,10) [shape=circle, fill=black] {};
\node (7') at (62.5,4) [shape=circle, fill=black] {};
\node (7ll) at (68,-3) [shape=circle, draw] {};
\node (7rl) at (72,-3) [shape=circle, draw] {};
\node (7+l) at (68,19) [shape=circle, fill=black] {};
\node (7+r) at (72,19) [shape=circle, fill=black] {};
\node (7w) at (70,5) [shape=circle, draw] {};
\node (7+w) at (70,16) [shape=circle, draw] {};
\node (7'w) at (62.5,0) [shape=circle, draw]{};
\node (7'l) at (60.5,-3) [shape=circle, fill=black] {};
\node (7'r) at (64.5,-3) [shape=circle, fill=black] {};
\node (7wu) at (73,6.5) [shape=circle, fill=black] {};
\node (7wl) at (73,3.5) [shape=circle, fill=black] {};

\draw [thick] (3l) to (3'w) to (4'w) to (55,0);

\draw [thick] (3'r) to (3'w) to (3'l);
\draw [thick, dashed] (29,-3) to (31,-3);
\draw [thick] (3rl) to (3) to (3ll);
\draw [thick, dashed] (36.5,-3) to (38.5,-3);
\draw [thick] (3+l) to (3+w) to (3+r);
\draw [thick, dashed] (29,19) to (31,19);
\draw [thick] (3wu) to (3w) to (3wl);
\draw [thick, dashed] (33,6) to (33,4);

\draw [thick] (4rl) to (4) to (4ll);
\draw [thick] (4'r) to (4'w) to (4'l);
\draw [thick, dashed] (44,-3) to (46,-3);
\draw [thick, dashed] (51.5,-3) to (53.5,-3);
\draw [thick] (4+l) to (4+w) to (4+r);
\draw [thick, dashed] (44,19) to (46,19);
\draw [thick] (4wu) to (4w) to (4wl);
\draw [thick, dashed] (48,6) to (48,4);

\draw [thick] (7rl) to (7) to (7ll);
\draw [thick] (7'r) to (7'w) to (7'l);
\draw [thick, dashed] (69,-3) to (71,-3);
\draw [thick, dashed] (61.5,-3) to (63.5,-3);
\draw [thick] (7+l) to (7+w) to (7+r);
\draw [thick, dashed] (69,19) to (71,19);
\draw [thick] (7wu) to (7w) to (7wl);
\draw [thick, dashed] (73,6) to (73,4);

\node (3+l) at (31,7) [shape=circle, draw] {};
\node (3+u) at (33,9) [shape=circle, draw] {};
\draw [thick] (3+) to (3+l);
\draw [thick] (3+) to (3+u);
\draw [thick, dashed] (31.6,7.6) to (32.5,8.5);

\node (4+l) at (46,7) [shape=circle, draw] {};
\node (4+u) at (48,9) [shape=circle, draw] {};
\draw [thick] (4+) to (4+l);
\draw [thick] (4+) to (4+u);
\draw [thick, dashed] (46.6,7.6) to (47.5,8.5);

\node (7+l) at (71,7) [shape=circle, draw] {};
\node (7+u) at (73,9) [shape=circle, draw] {};
\draw [thick] (7+) to (7+l);
\draw [thick] (7+) to (7+u);
\draw [thick, dashed] (71.6,7.6) to (72.5,8.5);

\draw [thick, dashed] (55,0) to (60,0);
\draw [thick] (60,0) to (7'w) to (75,0);
\draw [thick,dashed] (75,0) to (80,0);

\draw [thick] (3) to (3w) to (3+);
\draw [thick] (3'w) to (3');

\draw [thick] (4) to (4w) to (4+);
\draw [thick] (4'w) to (4');

\draw [thick] (7) to (7w) to (7+);
\draw [thick] (7'w) to (7');

\draw [thick] (29.6,16) arc (97:443:3);
\draw [thick] (44.6,16) arc (97:443:3);
\draw [thick] (69.6,16) arc (97:443:3);

\node at (29.6,-2) {$-1$};
\node at (28,-0.8) {$0$};
\node at (29.3,2.5) {$1$};
\node at (26,13) {$2$};
\node at (29.5,18) {$3$};
\node at (32,17.5) {$q$};
\node at (33.5,10.2) {$q+1$};
\node at (33.8,1) {$t-1$};
\node at (38,2.3) {$t$};
\node at (43.5,2.5) {$t+1$};
\node at (40,13) {$t+2$};
\node at (48,17.5) {$t+q$};
\node at (49.5,10) {$t+q+1$};
\node at (48.5,1) {$2t-1$};
\node at (53.5,2.3) {$2t$};
\node at (63.5,2.3) {$nt$};
\node at (68,2.5) {$nt+1$};
\node at (64.5,13) {$nt+2$};
\node at (68,2.5) {$nt+1$};
\node at (73.5,17.5) {$nt+q$};
\node at (75,10) {$nt+q+1$};
\node at (74.5,1) {$(n+1)t-1$};

\node at (57.5,15) {$F$};
\node at (30.5,22) {${\mathcal F}_0$};
\node at (37.5,6) {${\mathcal W}_1$};
\node at (45.5,22) {${\mathcal F}_1$};
\node at (52.5,6) {${\mathcal W}_2$};
\node at (62.5,6) {${\mathcal W}_n$};
\node at (70.5,22) {${\mathcal F}_n\;(n\ge 1)$};

\end{tikzpicture}
\end{center}
\caption{The bipartite map ${\mathcal N}_{p,q}$, with $t:=p+2q-2$} 
\label{mapNpq}
\end{figure}
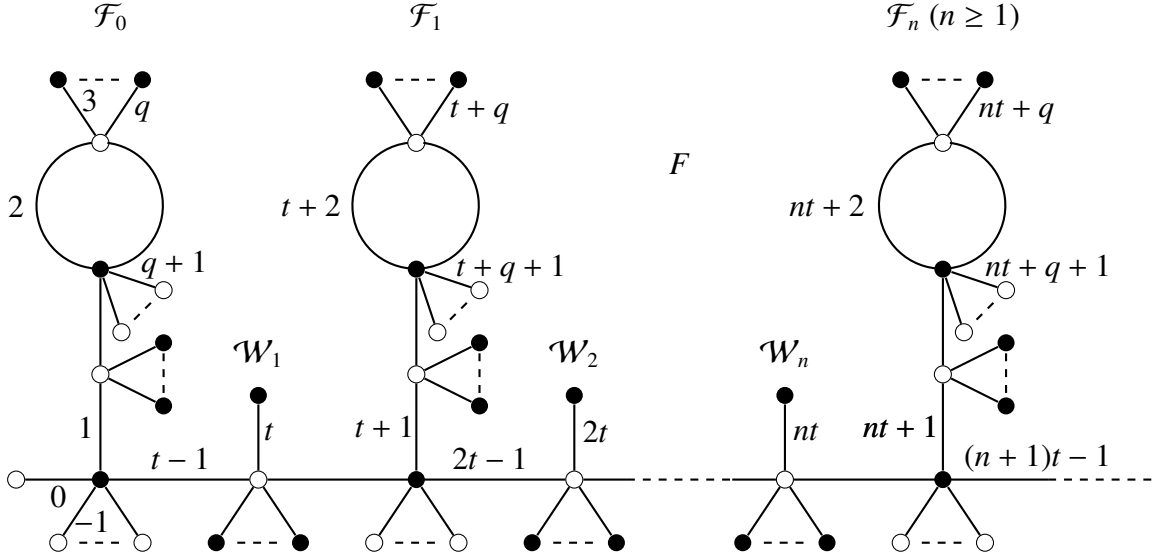

Note that while the flowers ${\mathcal F}_n$ have grown since Case~2 was proved, small `weeds' ${\mathcal W}_n\;(n\ge 1)$ have grown between them. This bipartite map is the Walsh map for an oriented hypermap of type $(p,q,\infty)$. Its monodromy group $G$ is generated by permutations $x$ and $y$, of order $p$ and $q$, which rotate edges around their incident black and white vertices. It is sufficient to show that $G$ acts primitively on the set $\Omega$ of edges, and that in the induced action of $\Gamma$ on $\Omega$, the subgroup $N$ fixing an edge has an epimorphism onto $F_{\infty}$.

The proof is similar to that for Case~2. The elements $\alpha z^i$ of the cycle $C$ of $z$ containing the edge $\alpha=0$ are labelled with  integers $i$. (To save space in Figure~\ref{mapNpq}, only a few significant labels are shown.) Any $\Gamma$-invariant equivalence relation $\sim$ restricts to $C$ as congruence mod~$(n)$ for some $n\in{\mathbb N}\cup\{\infty\}$. If $E=[\alpha]$ then since $\alpha y=\alpha$ we have $Ey=E$. If $n\ne\infty$ then $Ex=E$ since the edge $\beta\in E$ labelled $nt$ is fixed by $x$, so that $E\Gamma=E$ and hence $E=\Omega$. Thus we may assume that $n=\infty$, so all elements of $C$ are in distinct conjugacy classes and hence $E\cap C=\{\alpha\}$. Similarly $E\cap Cy=\{\alpha\}$. But $\Omega=C\cup Cy$, so $E=\{\alpha\}$ and $\sim$ is the identity relation. Thus $G$ is primitive, so the subgroup $N$ of $\Gamma$ fixing $\alpha$ is maximal. It is a free product of cyclic groups, of orders $p$, $q$ and $\infty$, corresponding to the fixed points of $x$, $y$ and $z$. There are infinitely many of each, and mapping those of finite order to the identity gives an epimorphism $N\to F_{\infty}$.

\medskip

\noindent{\sl Case 4: $p$ or $q=\infty$.} If $p=\infty$ or $q=\infty$ we can use the natural epimorphism from $\Gamma=\Delta(p,q,\infty)$ to a hyperbolic triangle group $\Gamma'=\Delta(p',q',\infty)$ with $p'$ and $q'$ both finite, use Case~1, 2 or 3 to establish the CARP for $\Gamma'$, and finally use Lemma~\ref{epi} to deduce it for $\Gamma$.
\end{proof}

Since the FARP has been dealt with, this completes the proof of Theorem~\ref{realisation}(b) for hyperbolic triangle groups $\Delta(p,q,r)$ with an infinite period.

\iffalse

Since we have already established the FARP for all hyperbolic triangle groups, we immediately deduce the following, which completes the proof of Theorem~\ref{realisation}{\color{red}(b)} for such groups and their associated categories of oriented hypermaps:

\begin{coro}\label{SARPcorol}
If $\Gamma$ is a hyperbolic triangle group $\Delta(p,q,r)$ with an infinite period, then $\Gamma$ and the corresponding category of oriented hypermaps have the strong automorphism realisation property.
\end{coro}

\fi

\begin{rema}
Constructions similar to those in the proof of Theorem~\ref{CpCq} have been used in~\cite{Jon18b} to prove that if $\Gamma$ is a hyperbolic triangle group with an infinite period then $\Gamma$ has uncountably many conjugacy classes of maximal subgroups of infinite index. This strengthens and generalises results of B.~H.~Neumann~\cite{Neu33}, Magnus~\cite{Mag73, Mag74}, Tretkoff~\cite{Tre}, and Brenner and Lyndon~\cite{BL} on maximal nonparabolic subgroups of the modular group, and has some overlap with work of Kulkarni~\cite{Kul} on maximal subgroups of free products of cyclic groups. We will use this in \S\ref{3c} to prove Theorem~\ref{realisation}(c).
\end{rema}

%%%%%%%%%%%%%%%%%%%%%%

\section{The CARP for extended triangle groups and unoriented hypermaps}\label{extended}

We now consider the CARP for extended triangle groups $\Gamma=\Delta[p,q,\infty]$ and their associated categories of unoriented hypermaps. In the preceding section we realised countable groups $A$ as automorphism groups in various categories $\mathfrak C^+$ of oriented hypermaps of a given type by constructing specific objects ${\mathcal N}={\mathcal N}_p\;(p\ge 3)$ or ${\mathcal N}_{p,q}\;(p, q\ge 3)$ in those categories, and then forming regular coverings $\mathcal M$ of $\mathcal N$, with covering group $A$, constructed so that $\mathcal M$ has only those automorphisms induced by $A$. These objects $\mathcal M$ and $\mathcal N$ correspond to subgroups $M$ and $N$ of the parent group $\Gamma^+=\Delta(p,q,\infty)$ for $\mathfrak C^+$ with $N=N_{\Gamma^+}(M)$. We can also regard $\mathcal M$ and $\mathcal N$ as objects in the corresponding category $\mathfrak C$ of unoriented maps or hypermaps of that type, for which the parent group is the extended triangle group $\Gamma$. 

\begin{lemm}\label{extlemma}
For these objects $\mathcal M$ we have ${\rm Aut}_{\mathfrak C}({\mathcal M})={\rm Aut}_{\mathfrak C^+}({\mathcal M})\cong A$.
\end{lemm}

\begin{proof}
We have ${\rm Aut}_{\mathfrak C}({\mathcal M})\cong N_{\Gamma}(M)/M$ and ${\rm Aut}_{\mathfrak C^+}({\mathcal M})\cong N_{\Gamma^+}(M)/M\cong A$ by Theorem~\ref{isothm}, so it is sufficient to show that $N_{\Gamma}(M)=N_{\Gamma^+}(M)$. Clearly $N_{\Gamma}(M)\ge N_{\Gamma^+}(M)$. If this inclusion is proper then since $N_{\Gamma^+}(M)=N_{\Gamma}(M)\cap\Gamma^+$ with $|\Gamma:\Gamma^+|=2$ we have $|N_{\Gamma}(M):N_{\Gamma^+}(M)|=2$, so the subgroup $N=N_{\Gamma^+}(M)$ is normalised by some elements of $\Gamma\setminus\Gamma^+$. This is impossible, since in all cases the map or hypermap $\mathcal N$ corresponding to $N$ is chiral (without orientation-reversing automorphisms), by the proof of Theorem~\ref{CpCq} and by inspection of Figures~\ref{mapN3}, \ref{mapNp} and~\ref{mapNpq}.
\end{proof}

\begin{coro}
If $(p,q,r)$ is a hyperbolic triple with an infinite period then the extended triangle group $\Delta[p,q,r]$ and the associated category of all hypermaps of type $(p,q,r)$ have the countable automorphism realisation property.
\end{coro}

\iffalse
\begin{proof}
Each of the infinitely many non-conjugate subgroups $M$ of $\Gamma^+$ with $N_{\Gamma^+}(M)/M\cong A$ is conjugate in $\Gamma$ to at most one other, so they contain infinitely many which are mutually non-conjugate in $\Gamma$, all with  $N_{\Gamma}(M)/M\cong A$.
\end{proof}
\fi

This completes the proof of Theorem~\ref{realisation}(b).

\begin{rema}
It would not have been possible to use Lemma~\ref{extlemma} also in the proof of Theorem~\ref{realisation}(a) in \S\ref{FARP}, since the maximal subgroups $N_n$ of $\Gamma^+=\Delta(p,q,r)$ constructed there are normalised by orientation-reversing elements of $\Gamma=\Delta[p,q,r]$. Instead of the natural representation of ${\rm PSL}_2(n)$, we could have used its representation on the cosets of a maximal subgroup $H\cong A_5$ for $n\equiv\pm 1$ mod~$(5)$, or $H\cong S_4$ for $n\equiv \pm 1$ mod~$(8)$: in both of these cases there are two conjugacy classes of subgroups $H$, transposed by conjugation in ${\rm PGL}_2(n)$ (see~\cite[Ch.~XII]{Dic}) and corresponding to a chiral pair of hypermaps. However, in either case the point stabilisers $H$ have constant order as $n\to\infty$, whereas Proposition~\ref{FARPprop} requires $|N_n:K_n|=|H|$ to be unbounded, so we would need an alternative argument to show that $M$ is not normal in $\Gamma^+$, as in the proof of Proposition~\ref{CARP}.
\end{rema}

%%%%%%%%%%%%%%%%%%%%%%%

\section{The proof of Theorem~\ref{realisation}(c)}
\label{3c}

In order to prove Theorem~\ref{realisation}(c) in the case of the FARP, we require each finite group $A$ to be isomorphic to ${\rm Aut}_{\mathfrak C}(\mathcal O)$ for $\aleph_0$ finite connected objects $\mathcal O$ in the relevant category $\mathfrak C$. These correspond to conjugacy classes of subgroups $M$ of finite index in the parent group $\Gamma$ of $\mathfrak C$. In the case of the triangle groups $\Gamma=\Delta(p,q,r)$ the result follows from Remark~\ref{FARPremark}, and for the extended triangle groups $\Delta[p,q,r]$ it then follows from Lemma~\ref{extlemma}.

For the CARP, with no finiteness restriction on $\mathcal O$ and $|\Gamma:M|$, we need $2^{\aleph_0}$ objects realising each countable group $A$. The proofs of the CARP for the various triangle groups $\Gamma=\Delta(p,q,\infty)$ in Theorem~\ref{CpCq} all used Proposition~\ref{CARP}, and by Remark~\ref{CARPremark} this yields the required number of objects provided $A\ne 1$. 

In fact for any $A$, the objects $\mathcal O$ produced by Proposition~\ref{CARP} are all regular coverings of the object $\mathcal N\cong \mathcal O/A\in\mathfrak C$ corresponding to the maximal subgroup $N$ of $\Gamma$. The proofs of the various cases of Theorem~\ref{CpCq}, where suitable subgroups $N$ of $\Gamma$ are constructed, can easily be adapted (as in~\cite{Jon18b}) to produce $2^{\aleph_0}$ conjugacy classes of subgroups $N$ satisfying the hypotheses of Proposition~\ref{CARP}; we thus obtain that many non-isomorphic objects $\mathcal N$, each with $2^{\aleph_0}$ coverings $\mathcal O$ realising a group $A\ne 1$, and one covering (namely $\mathcal O=\mathcal N$) realising $A=1$.

\iffalse
This adaptation is already implicit in the proof of Proposition~\ref{237prop}, where there are  $2^{\aleph_0}$ ways of pairing the infinitely many copies of $G$ to add bridges, and hence that is the number of possible maps $\mathcal N$ which can be used. In the case of Theorem~\ref{CpCq} there is a little more work to do.
\fi

We will give the required details in Case~1 of Theorem~\ref{CpCq}, where $p=3$, $q=2$ and $\Gamma$ is the modular group; the argument is similar in the other cases. We can modify the map $\mathcal N_3$ in Figure~\ref{mapN3} by adding `stalks' between the flowers $\mathcal F_n$, each consisting of a new vertex on the horizontal axis, and a new free edge pointing upwards. Adding a stalk between $\mathcal F_m$ and $\mathcal F_{m+1}$ adds $2$ to the value of all labels on flowers $\mathcal F_n$ for $n>m$. We need to preserve the property that only one of the two labels on the stem of each flower $\mathcal F_n\;(n>1)$ is divisible by $n$. This can be done, in $2^{\aleph_0}$ different ways, by ensuring that for each $n>1$ the total number of stalks added between $\mathcal F_1$ and $\mathcal F_n$ is a multiple of $n$. The proof for Case~1 then proceeds as before, except that it now yields $2^{\aleph_0}$ conjugacy classes of maximal subgroups $N$.

In the remaining cases of Theorem~\ref{CpCq} we could use similar modifications to the maps $\mathcal N_p$ and $
\mathcal N_{p,q}$ in Figures~\ref{mapNp} and \ref{mapNpq}, or alternatively add extra vertices and edges to those below the horizontal axis, so that the non-negative labels above the axis are unaltered. The results and proofs in \S\ref{extended} are not affected by these changes, so Theorem~\ref{realisation}(c) is also proved for the extended triangle groups $\Delta[p,q,r]$ with an infinite period. This completes the proof of Theorem~\ref{realisation}.

%%%%%%%%%%%%%%%%%%%%%%

\section{The CARP for some cocompact triangle groups}\label{cocompact}

In Theorem~\ref{realisation}(b) we proved the CARP only for those hyperbolic triangle groups $\Delta(p,q,r)$ and $\Delta[p,q,r]$ for which at least one period $p, q$ or $r$ is $\infty$. We would like to extend to this property to the cocompact case, where $p, q$ and $r$ are all finite. The arguments we used to prove the CARP depend on a standard generator of $\Delta(p,q,r)$ ($Z$, without loss of generality) having infinite order, so that it can have a cycle $C$ of infinite length in some permutation representation, which is then proved to be primitive by identifying $C$ with $\mathbb Z$. Clearly this is impossible in the cocompact case, so a different approach is needed. The following is a first step in this direction.

\begin{prop}\label{237prop}
If one of $p, q$ and $r$ is even, another is divisible by $3$, and the third is at least $7$, then the triangle groups $\Delta(p,q,r)$ and $\Delta[p,q,r]$ and their associated categories have the countable automorphism realisation property. 
\end{prop}

\begin{proof}
By permuting periods and applying Lemma~\ref{epi} we may assume that $p=3$, $q=2$ and $r\ge 7$. First suppose that $r=7$. We will construct an infinite transitive permutation representation of the group
\[\Gamma=\Delta[3,2,7]=\langle X, Y, T\mid X^3=Y^2=T^2=(XY)^7=(XT)^2=(YT)^2=1\rangle\]
(where $T=R_2$) in which the subgroup $\Gamma^+=\Delta(3,2,7)=\langle X, Y\rangle$ acts primitively, and we will then apply Proposition~\ref{CARP} to a point-stabiliser in $\Gamma^+$. This representation is constructed by adapting the Higman--Conder technique of `sewing coset diagrams together', used in~\cite{Con80} to realise finite alternating and symmetric groups as quotients of $\Gamma^+$ and $\Gamma$. We refer the reader to~\cite{Con80} for full technical details of this method. (Note that we have changed Conder's notation, which has $X^2=Y^3=1$, by transposing the symbols $X$ and $Y$; this has no significant effect on the following proof.)

\begin{figure}[h!]
\begin{center}
\begin{tikzpicture}[scale=0.45, inner sep=0.8mm]

\draw [thick] (2,5) arc (0:360:2);
\node (a1) at (2,5) [shape=circle, fill=black] {};
\node (a1') at (-2,5) [shape=circle, fill=black] {};
\node (b1) at (0,3) [shape=circle, fill=black] {};
\node (c1) at (0,4.5) [shape=circle, fill=black] {};

\draw [thick] (b1) to (c1);
\draw [thick] (1,5.5) to (c1) to (-1,5.5);

\draw [thick] (2,0) arc (0:360:2);
\node (a2) at (2,0) [shape=circle, fill=black] {};
\node (a2') at (-2,0) [shape=circle, fill=black] {};
\node (b2) at (0,-2) [shape=circle, fill=black] {};
\node (c2) at (0,-0.5) [shape=circle, fill=black] {};

\draw [thick] (b2) to (c2);
\draw [thick] (1,0.5) to (c2) to (-1,0.5);

\draw [thick] (2,-5) arc (0:360:2);
\node (a3) at (2,-5) [shape=circle, fill=black] {};
\node (a3') at (-2,-5) [shape=circle, fill=black] {};
\node (b3) at (0,-7) [shape=circle, fill=black] {};
\node (c3) at (0,-5.5) [shape=circle, fill=black] {};

\draw [thick] (b3) to (c3);
\draw [thick] (1,-4.5) to (c3) to (-1,-4.5);

\node (d) at (4,0) [shape=circle, fill=black] {};
\node (d') at (-4,0) [shape=circle, fill=black] {};
\draw [rounded corners, thick] (a1) to (4,5) to (d) to (4,-5) to (a3);
\draw [rounded corners, thick] (a1') to (-4,5) to (d') to (-4,-5) to (a3');
\draw [thick] (a2) to (d);
\draw [thick] (a2') to (d');

\node at (6, 0) {$\mathcal G$};

%%%%%%

\draw [thick] (17,5) arc (0:360:2);
\node (a1) at (17,5) [shape=circle, fill=black] {};
\node (a1') at (13,5) [shape=circle, fill=black] {};
\node (b1) at (15,3) [shape=circle, fill=black] {};
\node (c1) at (15,4.5) [shape=circle, fill=black] {};

\draw [thick] (b1) to (c1);
\draw [thick] (16,5.5) to (c1) to (14,5.5);

\draw [thick] (17,0) arc (0:360:2);
\node (a2) at (17,0) [shape=circle, fill=black] {};
\node (a2') at (13,0) [shape=circle, fill=black] {};
\node (c2) at (16,-1.75) [shape=circle, fill=black] {};
\node (c2') at (14,-1.75) [shape=circle, fill=black] {};
\node (d2) at (15,-2) [shape=circle, fill=black] {};

\draw [thick] (c2) to (16,0);
\draw [thick] (c2') to (14,0);

\node (e) at (19,0) [shape=circle, fill=black] {};
\node (e') at (11,0) [shape=circle, fill=black] {};
\node (f) at (17.5,-3.5) [shape=circle, fill=black] {};
\node (f') at (12.5,-3.5) [shape=circle, fill=black] {};
\node (g) at (15,-3.5) [shape=circle, fill=black] {};
\draw [thick] (g) to (d2);
\draw [rounded corners, thick] (a1) to (19,5) to (19,-3.5) to (11,-3.5) to (11,5) to (a1');
\draw [thick] (a2) to (e);
\draw [thick] (a2') to (e');
\draw [thick] (f) to (17.5,-2);
\draw [thick] (f') to (12.5,-2);

\node at (9, 0) {$\mathcal H$};

\end{tikzpicture}

\end{center}
\caption{The maps $\mathcal G$ and $\mathcal H$.}
\label{DessinsGH}
\end{figure}
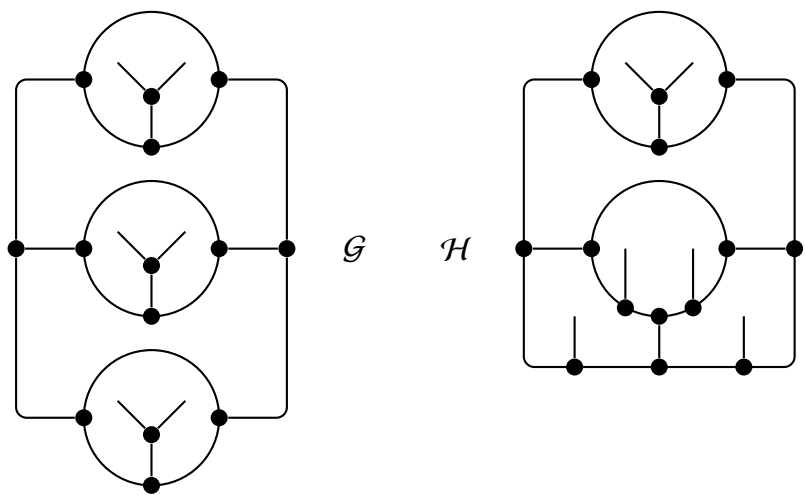

Conder gives $14$ coset diagrams $A,\ldots, N$ for subgroups of index $n=14,\ldots, 108$ in $\Gamma^+$, with respect to the generators $X$ and $Y$; these can be interpreted as describing transitive representations of $\Gamma^+$ of degree $n$. Each diagram is bilaterally symmetric, so this action of $\Gamma^+$ extends to a transitive representation of $\Gamma$ of degree $n$, with $T$ fixing vertices on the vertical axis of symmetry, and transposing pairs of vertices on opposite sides of it.  Although Conder does not do this, each of his diagrams can be converted into a planar map of type $\{7,3\}$ (equivalently a hypermap of type $(3,2,7)$), by contracting the small triangles representing $3$-cycles of $X$ to trivalent vertices, so that the cycles of $X, Y$ and $Z$ on directed edges correspond to the vertices, edges and faces, as before in this paper. (Warning: although $\Gamma^+$ acts as the monodromy group of this oriented map, permuting directed edges as described in Example~2 of \S\ref{permcats}, $\Gamma$ does {\sl not\/} act as the monodromy group of the unoriented map, as in Example~1: the latter permutes flags, whereas $T$ uses the symmetry of the map to extend the action of $\Gamma^+$ on directed edges to an action of $\Gamma$ on them.)

We will construct an infinite coset diagram from Conder's diagrams $G$ and $H$ of degree $n=42$; the corresponding maps $\mathcal G$ and $\mathcal H$ are shown in Figure~\ref{DessinsGH}. Conder defines a $(1)$-handle in a diagram to be a pair $\alpha, \beta$ of fixed points of $Y$ with $\beta=\alpha X=\alpha T$, represented in the corresponding map by two free edges incident with the same vertex on the axis of symmetry. Thus $\mathcal G$ has three $(1)$-handles, while $\mathcal H$ has one. If diagrams $D_i\; (i=1,2)$ of degree $n_i$ have $(1)$-handles $\alpha_i, \beta_i$ then one can form a new diagram, called a $(1)$-join $D_1(1)D_2$, by replacing these four fixed points of $Y$ with transpositions $(\alpha_1, \alpha_2)$ and $(\beta_1,\beta_2)$, and leaving the permutations representing $X, Y$ and $T$ in $D_1$ and $D_2$ otherwise unchanged; the result is a new coset diagram giving a transitive representation of $\Gamma$ of degree $n_1+n_2$. In terms of the corresponding maps $\mathcal D_i$, this is a connected sum operation, in which the two surfaces are joined across cuts between the free ends of the free edges representing the fixed points $\alpha_i$ and $\beta_i$; in particular, if $\mathcal D_i$ has genus $g_i$ then $\mathcal D_1(1)\mathcal D_2$ has genus $g_1+g_2$. (Further details about this and more general joining operations on dessins will appear in~\cite{JZ}.)

Using $(1)$-handles in $G$ and $H$, we first form an infinite diagram
\[H(1)G(1)G(1)G(1)G\cdots.\]
corresponding to an infinite planar map of type $\{7,3\}$. In this chain, each copy of $G$ has an unused $(1)$-handle; we join these arbitrarily in pairs, using $(1)$-compositions. Each such join adds a bridge to the underlying surface, so the result is an oriented trivalent map $\mathcal N$ of type $\{7,3\}$ and of infinite genus. This gives an infinite transitive permutation representation $X\mapsto x$, $Y\mapsto y$, $T\mapsto t$ of $\Gamma$ on the directed edges of $\mathcal N$, with $\Gamma^+$ again acting as its monodromy group, and $T$ acting as a reflection.

We need to prove that $\Gamma$ and $\Gamma^+$ act primitively. According to Conder~\cite{Con80} the permutation $w=yxt$ ($=xyt$ in his notation) induced by $YXT$ has cycle structures $1^313^3$ and $1^13^110^111^117^1$ in $G$ and $H$. In each of the $(1)$-compositions we have used, two fixed points of $w$ merge to form a cycle of length $2$ of $w$, and a cycle of $w$ of length $13$ in $G$ is merged with one of length $13$ or $10$ in $G$ or $H$ to form a cycle of length $26$ or $23$. All other cycles of $w$ are unchanged, so in particular its cycle $C$ of length $17$ in $H$ remains a cycle in the final diagram. Since all other cycles of $w$ have finite length coprime to $17$, some power of $w$ acts on $C$ as $w$ and fixes all other points. Since $17$ is prime, it follows that if $\Gamma^+$ acts imprimitively, then all points in $C$ must lie in the same equivalence class $E$. Now $C$ is what Conder calls a `useful cycle', since it contains a fixed point of $y$ not in a $(1)$-handle (the right-hand free edge $\beta$ in the central circle in $\mathcal H$ in Figure~\ref{DessinsGH}) and a pair of points from a $3$-cycle of $x$ (namely $\beta$ and $\beta x=\beta w^8$). It follows that $X$ and $Y$ leave $E$ invariant, which is impossible since they generate the transitive group $\Gamma^+$. Thus $\Gamma^+$ acts primitively (as therefore does $\Gamma$), so the point-stabilisers $N=\Gamma_{\alpha}$ and $N^+=\Gamma_{\alpha}^+$ of a directed edge $\alpha$ are maximal subgroups of $\Gamma$ and $\Gamma^+$. By the Reidemeister--Schreier algorithm, $N^+$ is a free product of four cyclic groups of order $2$ (arising from fixed points of $y$ in $H$ not in the $(1)$-handle), and infinitely many of infinite order (two arising from each bridge between a pair of copies of $G$). Thus $N^+$ admits an epimorphism onto $F_{\infty}$, so Proposition~\ref{CARP} shows that $\Gamma^+$ has the CARP.

We can choose $\alpha$ to be fixed by $t$, so that $T\in N$, and hence $N$ is a semidirect product of $N^+$ by $\langle T\rangle$. The action of $t$ is to reflect $H$ and all the copies of $G$ in the diagram, together with the bridges added between pairs of them. Acting by conjugation on $N^+$, $T$ therefore induces two transpositions on the elliptic generators of order $2$. Each bridge contributes a pair of free generators to $N^+$, one of them (representing a loop crossing the bridge and returning `at ground level'), centralised by $T$, the other (representing a loop transverse to the first, following a cross-section of the bridge) inverted by $T$; by sending $T$, together with the inverted generators and the four elliptic generators of $N^+$, to the identity, we can map $N$ onto the free group of infinite rank generated by the centralised generators, so Proposition~\ref{CARP} shows that $\Gamma$ has the CARP.

The extension to the case $r\ge 7$ is essentially the same, but based on the coset diagrams in Conder's later paper~\cite{Con81} on alternating and symmetric quotients of $\Delta(3,2,r)$ and $\Delta[3,2,r]$. In this case his coset diagrams $S(h,d)$ and $U(h,d)$ play the roles of $G$ and $H$, where $r=h+6d$ with $d\in{\mathbb N}$ and $h=7,\ldots, 12$.
\end{proof}

Note that this argument also shows that the groups and categories in Proposition~\ref{237prop} all satisfy Theorem~\ref{realisation}(c): there are $2^{\aleph_0}$ ways of pairing the copies of $G$ to produce bridges, giving mutually inequivalent representations and hence mutually non-conjugate subgroups $N$ and $N^+$ of $\Gamma$ and $\Gamma^+$.

\iffalse
It is hoped that full details of this proof will appear in a later paper.
\fi

Proposition~\ref{237prop} accounts for a proportion $121/216$ of all hyperbolic triples. It seems plausible that coset diagrams of Everitt~\cite{Eve} and others, constructed to extend Conder's results on alternating group quotients to all finitely generated non-elementary Fuchsian groups, could be used to prove  that all cocompact hyperbolic triangle groups $\Delta(p,q,r)$ and $\Delta[p,q,r]$, together with their associated categories, satisfy parts (b) and (c) of Theorem~\ref{realisation}, thus proving Conjecture~\ref{SARPconj}.

\section{Embedding all finite groups}\label{embedding}

Finally we return to the permutation group $P$ defined in the proof of Theorem~\ref{freeSARP} and discussed in Remark~\ref{gpP}, generated by the translation $x:i\mapsto i+1$ and the transposition $y=(0,1)$ of $\mathbb Z$. 
This is the monodromy group of the map $\mathcal M_{\infty}$ and of the corresponding hypermap $\mathcal H_{\infty}$ shown in Figure~\ref{Minfty}, where there are infinitely many edges. It follows that $P$ is isomorphic to the automorphism group ${\rm Aut}(\widetilde{\mathcal M}_{\infty})={\rm Aut}(\widetilde{\mathcal H}_{\infty})$ of the minimal regular covers $\widetilde{\mathcal M}_{\infty}$ and $\widetilde{\mathcal H}_{\infty}$ of these objects in their categories $\mathfrak M^+$ and $\mathfrak H^+$. Now for each $n\ge 1$ there is a subgroup of $P$, generated by the elements $y_k=y^{x^k}=(k,k+1)$ for $k=1,\ldots, n-1$, which is isomorphic to $S_n$. It therefore follows from Cayley's Theorem that, despite their very simple definition, the objects $\widetilde{\mathcal M}_{\infty}$ and $\widetilde{\mathcal H}_{\infty}$ are rich enough that their common automorphism group contains a copy of every finite group:

\begin{prop}\label{universal}
Every finite group is isomorphic to a subgroup of ${\rm Aut}_{\mathfrak M^+}(\widetilde{\mathcal M}_{\infty})={\rm Aut}_{\mathfrak H^+}(\widetilde{\mathcal H}_{\infty})$. 
\end{prop}

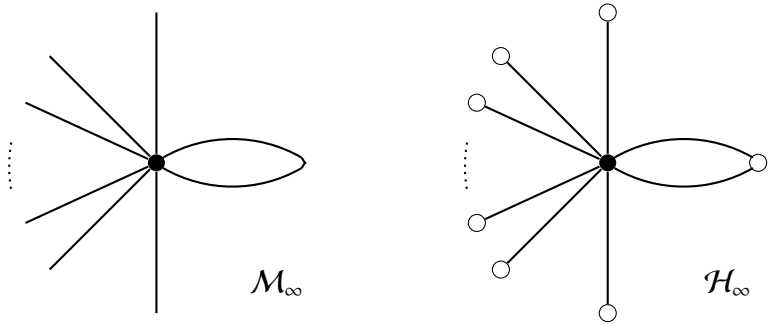
\begin{figure}[h!]
\begin{center}
\begin{tikzpicture}[scale=0.2, inner sep=0.8mm]

\node (A) at (0,0) [shape=circle, fill=black] {};
\draw [thick] (9.65,0.35) arc (60:120:9.2);
\draw [thick] (9.65,-0.35) arc (-60:-120:9.2);
\draw [thick] (9.65,0.35) to (9.75,0.2) to (9.8,0.1) to (9.9,0) to (9.8,-0.1) to (9.75,-0.2) to (9.65,-0.35);
\draw [thick] (0,10) to (A) to (0,-10);
\draw [thick] (-7.1,7.1) to (A) to (-7.1,-7.1);
\draw [thick] (-8.7,4) to (A) to (-8.7,-4);
\draw [thick, dotted] (-9.65,1.35) arc (170:190:9.2);

\node at (8,-8) {${\mathcal M}_{\infty}$};

%%%%%%%

\node (A') at (30,0) [shape=circle, fill=black] {};
\node (0') at (40,0) [shape=circle, draw] {};
\node (1') at (30,10) [shape=circle, draw] {};
\node (-1') at (30,-10) [shape=circle, draw] {};
\draw [thick] (39.65,0.35) arc (60:120:9.2);
\draw [thick] (39.65,-0.35) arc (-60:-120:9.2);
\node (2') at (22.9,7.1) [shape=circle, draw] {};
\node (-2') at (22.9,-7.1) [shape=circle, draw] {};
\node (3') at (21.3,4) [shape=circle, draw] {};
\node (-3') at (21.3,-4) [shape=circle, draw] {};
\draw [thick] (1') to (A') to (-1');
\draw [thick] (2') to (A') to (-2');
\draw [thick] (3') to (A') to (-3');
\draw [thick, dotted] (20.65,1.35) arc (170:190:9.2);

\node at (38,-8) {${\mathcal H}_{\infty}$};

\end{tikzpicture}
\end{center}
\caption{The map ${\mathcal M}_{\infty}$ and the hypermap ${\mathcal H}_{\infty}$} 
\label{Minfty}
\end{figure}

A slight modification of this construction, redefining $y$, yields a similar result for certain other categories of hypermaps. First we note that the group $F={\rm FSym}({\mathbb Z})$ of finitary permutations of $\mathbb Z$ has a subgroup of index $2$, the group  $F^+={\rm FSym}^+({\mathbb Z})$ of even finitary permutations of $\mathbb Z$. Similarly, $P=F\rtimes\langle x\rangle$ has an `even' subgroup $P^+=F^+\rtimes\langle x\rangle$ of index $2$. We now need the following lemma:

\begin{lemm}\label{xandy}
Let $q\ge 2$. The permutations $x$ and $y$ of $\mathbb Z$ defined by $x:i\mapsto i+1$ and $y=(1,2,\ldots, q)$ generate the group $P$ if $q$ is even, and $P^+$ if $q$ is odd.
\end{lemm}

\begin{proof}
The normal closure $N$ of $y$ in $P$ is a subgroup of $F$, and of $F^+$ if $q$ is odd, containing the permutation
\[y^xy^{-1}=(2,3,\ldots,q+1)(q,q-1,\ldots, 1)=(1,q,q+1).\]
Conjugating this by $(y^x)^2=(2,3,\ldots,q+1)^2$ gives $(1,2,3)\in N$, and conjugating this by $x^{k-1}$ gives $(k,k+1,k+2)\in N$ for all $k\in{\mathbb Z}$. Now $\langle(1,2,3),\ldots, (n-2, n-1, n)\rangle=A_n$ for each $n\ge 3$, so the $3$-cycles $(k,k+1,k+2)$ for $k\in\mathbb Z$ generate $F^+$. Thus $N=F$ or $F^+$ as $n$ is even or odd, giving $\langle x, y\rangle=N\rtimes\langle x\rangle=P$ or $P^+$ respectively. 
\end{proof}

Thus for each $q\ge 2$ there is an epimorphism $\Delta(\infty, q, \infty)\to P$ or $P^+$ given by $X\mapsto x, Y\mapsto y$ as $q$ is even or odd. If $\mathcal H_q$ is the corresponding oriented hypermap of type $(\infty, q, \infty)$ with monodromy group $P$ or $P^+$, its minimal regular cover $\widetilde{\mathcal H}_q$ has automorphism group $P$ or $P^+$ respectively. Since $S_n$ can be embedded in $A_{n+2}$ (as the stabiliser of $\{n+1,n+2\}$), we therefore have the following:

\begin{coro}
For each $q\ge 2$ there is an oriented hypermap of type $(\infty, q, \infty)$ such that its automorphism group contains an isomorphic copy of every finite group.
\end{coro}

%%%%%%%%%%%%%%%%%%
%%%%%%%%%%%%%%%%%%

\end{document}